\documentclass[12pt]{amsart}
\usepackage{amsfonts, amssymb}
\usepackage{parskip, fullpage, verbatim}
\usepackage[colorlinks=true]{hyperref}
\usepackage{breakurl}
\newcommand\CA{{\mathcal A}} 
\newcommand\CB{{\mathcal B}}
 
\newcommand\CD{{\mathcal D}}
 
\newcommand\CF{{\mathcal F}}
 
\newcommand\CIF{{\mathcal {IF}}} 
\newcommand\CFac{{\mathcal {F\!AC}}} 
\newcommand\CIFAC{{\mathcal {IF\!AC}}} 
\newcommand\CHFAC{{\mathcal {HF\!AC}}} 
\newcommand\CHIFAC{{\mathcal {HIF\!AC}}} 

\newcommand\RF{{\mathcal {RF}}}  
\newcommand\CSS{{\mathcal {SS}}}

\newcommand\R{{\varrho}}

\newcommand\BBC{{\mathbb C}}
\newcommand\BBK{{\mathbb K}}
\newcommand\BBQ{{\mathbb Q}}

\newcommand\BBZ{{\mathbb Z}}
\newcommand\BBF{{\mathbb F}}

\newcommand\codim{\operatorname{codim}}

\newcommand\Poin{{\operatorname{Poin}}}

\numberwithin{equation}{section}

\theoremstyle{plain}
\newtheorem{lemma}[equation]{Lemma}
\newtheorem{theorem}[equation]{Theorem}

\newtheorem{corollary}[equation]{Corollary}
\newtheorem{proposition}[equation]{Proposition}
\theoremstyle{definition}
\newtheorem{defn}[equation]{Definition}
\newtheorem{remark}[equation]{Remark}
\newtheorem{example}[equation]{Example}

\thanks{We acknowledge 
support from the DFG-priority program 
SPP1489 ``Algorithmic and Experimental Methods in
Algebra, Geometry, and Number Theory''.}

\subjclass[2010]{Primary 52B30, 52C35, 14N20; Secondary 51D20}

\begin{document}

%%%%%%%%%%%%%%%%%%%%%%%%%%%%%%%%%%%%%%%%%%%%%%%%%%%%%%%%%%%%%%%%%%%%%%
%%%%%%%%%%%%% top matter stuff
%%%%%%%%%%%%%%%%%%%%%%%%%%%%%%%%%%%%%%%%%%%%%%%%%%%%%%%%%%%%%%%%%%%%%%
\title[Addition-Deletion Theorems for Factorizations of Orlik-Solomon Algebras]
{Addition-Deletion Theorems for Factorizations of Orlik-Solomon Algebras and
nice Arrangements}

\author[T. Hoge]{Torsten Hoge}
\address
{Institut f\"ur Algebra, Zahlentheorie und Diskrete Mathematik,
Fakult\"at f\"ur Mathematik und Physik,
Leibniz Universit\"at Hannover,
Welfengarten 1,
30167 Hannover, Germany}
\email{hoge@math.uni-hannover.de}

\author[G. R\"ohrle]{Gerhard R\"ohrle}
\address
{Fakult\"at f\"ur Mathematik,
Ruhr-Universit\"at Bochum,
D-44780 Bochum, Germany}
\email{gerhard.roehrle@rub.de}

\keywords{
Orlik-Solomon algebra,
supersolvable arrangement,
nice arrangement,
inductively factored arrangement,
free and inductively free arrangement}

\allowdisplaybreaks

\begin{abstract}
We study the notion of a 
nice partition or 
factorization 
of a hyperplane  arrangement 
due to Terao
from the early 1990s. 
The principal aim of this note is 
an analogue of 
Terao's celebrated
addition-deletion theorem for 
free arrangements for the class of 
nice arrangements.
This is a natural setting for the
stronger property of an inductive factorization
of a hyperplane  arrangement 
by Jambu and Paris.

In addition, we show that 
supersolvable arrangements are 
inductively factored and that 
inductively factored 
arrangements are 
inductively free.
Combined with 
our addition-deletion theorem 
this leads to the concept of an 
induction table for 
inductive factorizations.

Finally, we prove that 
the notions of
factored and inductively 
factored arrangements are compatible with 
the product construction for arrangements.
\end{abstract}

\maketitle

%%%%%%%%%%%%%%%%%%%%%%%%%%%%%%%%%%%%%%%%%%%%%%%%%%%%%%%%%%%%%%%%%%%%%%
%%%%%%%%%%%%% article body...
%%%%%%%%%%%%%%%%%%%%%%%%%%%%%%%%%%%%%%%%%%%%%%%%%%%%%%%%%%%%%%%%%%%%%%

%%%%%%%%%%%%%%%%%%%%%%%%%%%%%%%%%%%%%%%%%%%%%%%%%%%%%%%%%%%%%%%%%%%%%%
%%%%%%%%%%%%% \S1 Introduction
%%%%%%%%%%%%%%%%%%%%%%%%%%%%%%%%%%%%%%%%%%%%%%%%%%%%%%%%%%%%%%%%%%%%%%
\section{Introduction}

Let $\BBK$ be a field and let $V = \BBK^\ell$.
Let $\CA = (\CA,V)$ be a central $\ell$-arrangement of hyperplanes 
in $V$.
The most basic algebraic invariant associated with an arrangement $\CA$ is its
so called \emph{Orlik-Solomon algebra} $A(\CA)$, introduced by Orlik and Solomon in \cite{orliksolomon:hyperplanes}.
The $\BBK$-algebra $A(\CA)$ is a graded and anti-commutative.
It is generated by $1$ in degree $0$ and by 
a set of degree $1$ generators $\{a_H \mid  H \in \CA\}$, e.g.\ see 
\cite[\S 3.1]{orlikterao:arrangements}.
Let $A(\CA) = \oplus_{i=0}^r A(\CA)_i$ be the decomposition of
$A(\CA)$ into homogeneous components, 
so that $\Poin(A(\CA), t) = \sum_{i=0}^r (\dim A(\CA)_i)t^i$,
where $r = r(\CA)$ is the rank of $\CA$.
In particular, $A(\CA)_0 = \BBK$ and  $A(\CA)_1 = \sum_{H\in \CA} \BBK a_H$.

Thanks to a fundamental result due to Orlik and Solomon
\cite[Thm.\ 2.6]{orliksolomon:hyperplanes} 
(cf.\ \cite[Thm.\ 3.68]{orlikterao:arrangements}),
the Poincar\'e polynomial of $A(\CA)$ 
coincides with the combinatorially defined Poincar\'e polynomial 
$\pi(\CA,t)$ of $\CA$,
\begin{equation}
\label{eq:piA}
\Poin(A(\CA), t) = \pi(\CA,t).
\end{equation}

The geometric significance of $A(\CA)$ stems 
from the fact that in case $\BBK = \BBC$ is the field of complex numbers,
Orlik and Solomon showed in 
\cite[Thm.\ 5.2]{orliksolomon:hyperplanes} 
that as an associative, graded $\BBC$-algebra $A(\CA)$
is isomorphic to 
the cohomology algebra of the complement $M(\CA)$ of 
the complex arrangement $\CA$ 
(cf.\ \cite[\S 5.4]{orlikterao:arrangements}):
\[
A(\CA) \cong H^*(M(\CA)).
\]
In particular, the Poincar\'e polynomial $\Poin(M(\CA), t)$ of $M(\CA)$
is given by $\Poin(A(\CA), t)$.

Let $\pi = (\pi_1, \ldots, \pi_s)$ be a partition of $\CA$ and let 
\[
[\pi_i] := \BBK + \sum_{H\in \pi_i} \BBK a_H
\] 
be the $\BBK$-subspace of $A(\CA)$ 
spanned by $1$ and the set of $\BBK$-algebra generators $a_H$ 
of $A(\CA)$ corresponding to the members in $\pi_i$.
So the Poincar\'e polynomial of the graded $\BBK$-vector space $[\pi_i]$ is just  
$\Poin([\pi_i],t) = 1 + | \pi_i| t$.
Consider the canonical $\BBK$-linear map 
\begin{equation}
\label{eq:factoredOS}
\kappa : [\pi_1] \otimes \cdots \otimes [\pi_s] \to A(\CA)
\end{equation}
given by multiplication. We say that $\pi$ gives
rise to a \emph{tensor factorization} of $A(\CA)$ 
if $\kappa$ is an isomorphism of graded $\BBK$-vector spaces. In this case
$s = r$, as $r$ is the top degree of $A(\CA)$, and thus we get 
a factorization of the Poincar\'e polynomial of $A(\CA)$ into linear terms
\begin{equation}
\label{eq:poinOS}
\Poin(A(\CA),t) = \prod_{i=1}^r (1 + |\pi_i| t).
\end{equation}
For $\CA = \Phi_\ell$ the empty arrangement, 
we set $[\varnothing] := \BBK$, so that 
$\kappa : [\varnothing] \cong A(\Phi_\ell)$.

In \cite[Thm.\ 5.3]{orliksolomonterao:hyperplanes},
Orlik, Solomon and Terao showed that 
a supersolvable arrangement $\CA$ admits a partition 
$\pi$  which gives rise to a tensor factorization of $A(\CA)$
via $\kappa$  in \eqref{eq:factoredOS} 
(cf.\ \cite[Thm.\ 3.81]{orlikterao:arrangements}).
Jambu gave an alternate proof of this fact 
in \cite[Prop.\ 3.2.2]{jambu:factored}.
In \cite{bjoernerziegler:factored},
Bj\"orner and Ziegler gave a sufficient condition for such 
factorizations of $A(\CA)$.

In \cite{terao:factored}, Terao was able to 
capture this tensor factorization property of $A(\CA)$ 
as in \eqref{eq:factoredOS}
purely combinatorially in terms of the 
underlying partition $\pi$, as follows.
Let $L(\CA)$ be the intersection lattice of $\CA$.
To each $X$ in $L(\CA)$ we associate
the subarrangement  $\CA_X$  of $\CA$,
where 
$\CA_X  =\{H \in \CA \mid X \subset H\}$.
Let $\pi = (\pi_1, \ldots, \pi_s)$ be a partition of $\CA$.
Following \cite{terao:factored}, 
$\pi$ is called \emph{nice} for $\CA$ or a \emph{factorization} of $\CA$ 
if firstly $\pi$ is \emph{independent}, i.e.\
for any choice $H_i \in \pi_i$ for $1 \le i \le s$,
the resulting $s$ hyperplanes are linearly independent, 
and secondly, for each $X$ in $L(\CA)$, 
the \emph{induced partition} $\pi_X$ of   
$\CA_X$ consisting of the non-empty blocks
of the form $\pi_i \cap \CA_X$ admits a singleton as one of its parts, 
see Definition \ref{def:factored}. 
We also say that $\CA$ is \emph{nice} or \emph{factored}
provided $\CA$ admits a nice partition.

In \cite[Thm.\ 2.8]{terao:factored},
Terao proved that 
$\pi$ gives
rise to a tensor factorization of the Orlik-Solomon algebra $A(\CA)$
via $\kappa$ 
as in \eqref{eq:factoredOS}
if and only if 
$\pi$ is 
nice for $\CA$, see Theorem \ref{thm:teraofactored}
(cf.\ \cite[Thm.\ 3.87]{orlikterao:arrangements}).
Note that $\kappa$ is not an isomorphim of $\BBK$-algebras.

In order to state our principal results, 
we need a bit more notation.
Suppose that $\CA$ is non-empty and 
let $\pi = (\pi_1, \ldots, \pi_s)$ be a  partition  of $\CA$.
Let $H_0 \in \pi_1$ and let
$(\CA, \CA', \CA'')$ be the triple associated with $H_0$. 
We have the \emph{induced partition} 
$\pi'$ of $\CA'$
consisting of the non-empty parts $\pi_i' := \pi_i \cap \CA'$.
Further, we have the \emph{restriction map} 
$\R = \R_{\pi,H_0} : \CA \setminus \pi_1 \to \CA''$ given by 
$H \mapsto H \cap H_0$, depending on $\pi$ and $H_0$.
Let $\pi_i'' := \R(\pi_i)$ for $i = 2, \ldots, s$.
Clearly, imposing that 
$\pi'' = (\pi''_2, \ldots, \pi''_s)$ is
again a partition of $\CA''$ entails that 
$\R$ is onto.
It turns out that the injectivity of $\R$ is the key condition
in our context. 

Our chief result for Orlik-Solomon algebras is
the following Addition-Deletion Theorem for tensor 
factorizations of $A(\CA)$ of the form \eqref{eq:factoredOS}.
Analogous to $\kappa$ in \eqref{eq:factoredOS},
$\kappa'$ and $\kappa''$ are the $\BBK$-linear maps 
determined  by the partitions $\pi'$ and $\pi''$ of 
$\CA'$ and $\CA''$, respectively.

\begin{theorem}
\label{theorem:add-del-OS}
Suppose that $\CA \ne \Phi_\ell$. 
Let $\pi = (\pi_1, \ldots, \pi_s)$ be a  partition  of $\CA$.
Let $H_0 \in \pi_1$ and let
$(\CA, \CA', \CA'')$ be the triple associated with $H_0$. 
Suppose that $\R: \CA \setminus \pi_1 \to \CA''$ 
is bijective.
Then any two of the following statements imply the third:
\begin{itemize}
\item[(i)] $\kappa : [\pi_1] \otimes \cdots \otimes [\pi_s] \to  A(\CA)$
is an isomorphism of graded $\BBK$-vector spaces;
\item[(ii)] $\kappa' : [\pi_1'] \otimes \cdots \otimes [\pi_s']  \to  A(\CA')$
is an isomorphism of graded $\BBK$-vector spaces;
\item[(iii)] $\kappa'' : [\pi_2''] \otimes \cdots \otimes [\pi_s'']  \to  A(\CA'')$
is an isomorphism of graded $\BBK$-vector spaces.
\end{itemize}
\end{theorem}

Theorem \ref{theorem:add-del-OS} 
is an immediate consequence of
Terao's Theorem \ref{thm:teraofactored} 
and the following result which is 
an analogue of Terao's celebrated
Addition-Deletion Theorem \ref{thm:add-del} for 
free arrangements for the class of 
nice arrangements.

\begin{theorem}
\label{theorem:add-del-factored}
Suppose that $\CA \ne \Phi_\ell$. 
Let $\pi = (\pi_1, \ldots, \pi_s)$ be a  partition  of $\CA$.
Let $H_0 \in \pi_1$ and let
$(\CA, \CA', \CA'')$ be the triple associated with $H_0$. 
Suppose that $\R: \CA \setminus \pi_1 \to \CA''$ 
is bijective.
Then any two of the following statements imply the third:
\begin{itemize}
\item[(i)] $\pi$ is nice for $\CA$;
\item[(ii)] $\pi'$ is nice for $\CA'$;
\item[(iii)] $\pi''$ is nice for $\CA''$.
\end{itemize}
\end{theorem}

As indicated above, 
if $\CA$ is supersolvable, then $\CA$ 
satisfies the factorization property 
from \eqref{eq:factoredOS}, so $\CA$ is nice. 
Nevertheless, there is no counterpart of the 
Addition-Deletion Theorem \ref{theorem:add-del-factored} 
for the stronger notion of 
supersolvable arrangements,
see Example \ref{ex:d13}.

By Terao's \emph{Factorization Theorem} \ref{thm:freefactors},
the Poincar\'e polynomial 
of a free arrangement $\CA$
factors into linear terms 
given by the exponents of $\CA$, i.e.\ 
\begin{equation}
\label{eq:pifree}
\pi(\CA,t) = \prod_{i=1}^\ell (1 + b_i t),
\end{equation}
where $\exp \CA = \{b_1, \ldots, b_\ell\}$ are the exponents of $\CA$,
\cite{terao:freefactors}
(cf.\ \cite[Thm.\ 4.137]{orlikterao:arrangements}).
In view of \eqref{eq:piA},
considering both factorizations of the 
Poincar\'e polynomials given in 
\eqref{eq:poinOS} and \eqref{eq:pifree},
it is natural to ask whether every nice arrangement is
free, \cite{terao:factored}.  This however is not the case. It
was observed by Enta, Falk and Ziegler independently
that the $3$-arrangement in characteristic $3$ in  
\cite[Ex.\ 4.1]{ziegler} is factored but not free
(cf.\ note added in proof in \cite{terao:factored}).
Vice versa, a free arrangement need not be factored either, e.g.\ 
the reflection arrangment of 
the Coxeter group of type $D_4$
is free but not factored, see Remark \ref{rem:indfactored}(ii).

Combining Theorem \ref{theorem:add-del-factored} with Terao's
Addition-Deletion Theorem \ref{thm:add-del} for free arrangements, 
we get an
Addition-Deletion Theorem for 
the proper subclass of 
nice and free arrangements.

\begin{theorem}
\label{theorem:add-del-free-factored2}
Suppose that $\CA \ne \Phi_\ell$.
Let $\pi = (\pi_1, \ldots, \pi_s)$ be a partition  of $\CA$.
Let $H_0 \in \pi_1$ and 
let $(\CA, \CA', \CA'')$ be the triple associated with $H_0$. 
Suppose that $\R: \CA \setminus \pi_1 \to \CA''$ 
is bijective. 
Then any two of the following statements imply the third:
\begin{itemize}
\item[(i)] $\pi$ is nice for $\CA$ and $\CA$ is free; 
\item[(ii)] $\pi'$  is nice for $\CA'$ and $\CA'$ is free; 
\item[(iii)] $\pi''$  is nice for $\CA''$ and $\CA''$ is free. 
\end{itemize}
\end{theorem}

It is worth noting that
in Theorem \ref{theorem:add-del-free-factored2}
we do not need to explicitly impose the 
containment conditions on the sets of exponents of the arrangements involved,
cf.\ Theorem \ref{thm:add-del}. This is a consequence of the
presence of the underlying factorizations along with the 
bijectivity condition on $\R$.
We also note that this condition is necessary, 
see Example \ref{ex:a222}.

Theorems \ref{theorem:add-del-OS},
\ref{theorem:add-del-factored}, and 
\ref{theorem:add-del-free-factored2}
likely prove to be equally viable 
for Orlik-Solomon algebras and 
nice arrangements, as  is
Terao's Addition-Deletion Theorem \ref{thm:add-del} 
for free arrangements.
In Examples \ref{ex:ot454}, \ref{ex:d13}, \ref{ex:g333},
\ref{ex:indfreefactored-notindfactored} 
and \ref{ex:notheredfactored},
we demonstrate the usefulness 
of these results.

Theorem \ref{theorem:add-del-factored} naturally motivates the 
notion of an inductive factorization, due to Jambu and Paris \cite{jambuparis:factored}.
In our setting 
an arrangement $\CA$ is called \emph{inductively factored}
provided 
there exists a partition $\pi$ of $\CA$ and a hyperplane $H_0 \in \pi_1$ 
such that for the triple $(\CA, \CA', \CA'')$ associated with $H_0$ 
the induced partition $\pi'$ of $\CA'$ is an inductive factorization, 
the associated restriction map $\R: \CA \setminus \pi_1 \to \CA''$ 
is bijective and the induced partition 
$\pi''$ of $\CA''$ is an inductive factorization of $\CA''$,
see Definition \ref{def:indfactored}.

In Proposition \ref{prop:indfactoredindfree}
we show that every inductively factored 
arrangement is inductively free.
It follows that an 
inductively factored arrangement admits 
an induction table of inductively free subarrangements.
In Remark \ref{rem:indtable}, we show that an 
inductive factorization 
can be achieved by means of an 
\emph{induction of factorizations} procedure 
in form of an \emph{induction table of factorizations} 
which extends the method of 
induction of hyperplanes for inductively free arrangements.
In Section \ref{subsec:examples}, we 
demonstrate this method
with several examples.

Terao  \cite{terao:factored} showed that 
every supersolvable arrangement is factored, 
see Proposition \ref{prop:ssfactored}.
Indeed, every supersolvable arrangement is inductively factored,
see Proposition \ref{prop:ssindfactored}.
Moreover, Jambu and Paris showed that each inductively factored 
arrangement is inductively free, see Proposition \ref{prop:indfactoredindfree}
(\cite[Prop.\ 2.2]{jambuparis:factored}). 
Each of these classes of arrangements is properly contained in the other,
see Remark \ref{rem:inclusions}.

The paper is organized as follows.
In the next section we recall the required notions of 
free, inductively free, 
supersolvable and nice arrangements 
mostly taken from
\cite{orlikterao:arrangements} and 
\cite{terao:factored}.
In Section \ref{ssect:factored}, we 
recall the main results from \cite{terao:factored}.

In Section \ref{sec:adddel} we prove 
slightly stronger versions of 
Theorems \ref{theorem:add-del-factored} 
and~\ref{theorem:add-del-free-factored2}.

Subsequently, in Section \ref{ssec:indfactoredindfree}, we
introduce the notion of an inductively 
factored arrangement due to Jambu and Paris,
\cite{jambuparis:factored}.
In Proposition \ref{prop:ssindfactored},
we show that every supersolvable arrangement is 
inductively factored, 
and in Proposition \ref{prop:indfactoredindfree}
that every inductively factored 
arrangement is inductively free.
These results in turn are extended to 
hereditarily inductively factored arrangements in 
Section \ref{sec:heredfactored}.

In Remark \ref{rem:indtable}, we introduce the 
concept of induction of factorizations 
for inductively factored arrangements.
In Section \ref{subsec:examples}, we
present applications of our main results.

In Propositions \ref{prop:product-factored}
and \ref{prop:product-indfactored},
we show that factored and
inductively factored arrangements 
are compatible with the product 
construction for arrangements,
as is the case for free,  
inductively free and 
supersolvable arrangements,
see Propositions \ref{prop:product-free},  
\ref{prop:product-indfree},  and 
\cite[Prop.\ 2.6]{hogeroehrle:super}, 
respectively. Moreover, we 
extend this compatibility to 
hereditarily (inductively) factored arrangements 
in Corollary \ref{cor:product-heredindfactored}.

For general information about arrangements 
we refer the reader to \cite{orlikterao:arrangements}.

\section{Recollections and Preliminaries}

\subsection{Hyperplane Arrangements}
\label{ssect:hyper}
Let $V = \BBK^\ell$ 
be an $\ell$-dimensional $\BBK$-vector space.
A \emph{hyperplane arrangement} is a pair
$(\CA, V)$, where $\CA$ is a finite collection of hyperplanes in $V$.
Usually, we simply write $\CA$ in place of $(\CA, V)$.
We write $|\CA|$ for the number of hyperplanes in $\CA$.
The empty arrangement in $V$ is denoted by $\Phi_\ell$.

The \emph{lattice} $L(\CA)$ of $\CA$ is the set of subspaces of $V$ of
the form $H_1\cap \dotsm \cap H_i$ where $\{ H_1, \ldots, H_i\}$ is a subset
of $\CA$. 
For $X \in L(\CA)$, we have two associated arrangements, 
firstly the subarrangement 
$\CA_X :=\{H \in \CA \mid X \subseteq H\} \subseteq \CA$
of $\CA$ and secondly, 
the \emph{restriction of $\CA$ to $X$}, $(\CA^X,X)$, where 
$\CA^X := \{ X \cap H \mid H \in \CA \setminus \CA_X\}$.
Note that $V$ belongs to $L(\CA)$
as the intersection of the empty 
collection of hyperplanes and $\CA^V = \CA$. 
The lattice $L(\CA)$ is a partially ordered set by reverse inclusion:
$X \le Y$ provided $Y \subseteq X$ for $X,Y \in L(\CA)$.

If $0 \in H$ for each $H$ in $\CA$, then 
$\CA$ is called \emph{central}.
If $\CA$ is central, then the \emph{center} 
$T_\CA := \cap_{H \in \CA} H$ of $\CA$ is the unique
maximal element in $L(\CA)$  with respect
to the partial order.
We have a \emph{rank} function on $L(\CA)$: $r(X) := \codim_V(X)$.
The \emph{rank} $r := r(\CA)$ of $\CA$ 
is the rank of a maximal element in $L(\CA)$.
The $\ell$-arrangement $\CA$ is \emph{essential} 
provided $r(\CA) = \ell$.
If $\CA$ is central and essential, then $T_\CA =\{0\}$.
Throughout, we only consider central arrangements.

The \emph{Poincar\'e polynomial} $\pi(\CA,t) \in \BBZ[t]$ 
of $\CA$ is defined by 
\[
\pi(\CA,t) := \sum_{X \in L(\CA)} \mu(X)(-t)^{r(X)},
\]
where $\mu$ is the M\"obius function of $L(\CA)$, 
see \cite[Def.\ 2.48]{orlikterao:arrangements}.

Let $S = S(V^*)$ be the symmetric algebra of the dual space $V^*$ of $V$.
If $x_1, \ldots , x_\ell$ is a basis of $V^*$, then we identify $S$ with 
the polynomial ring $\BBK[x_1, \ldots , x_\ell]$.
For $H \in \CA$ we fix $\alpha_H \in V^*$ with $H = \ker \alpha_H$.
The \emph{defining polynomial} $Q(\CA)$ of $\CA$ is given by 
$Q(\CA) := \prod_{H \in \CA} \alpha_H \in S$.

For $\CA \ne \Phi_\ell$, 
let $H_0 \in \CA$.
Define $\CA' := \CA \setminus\{ H_0\}$,
and $\CA'' := \CA^{H_0} = \{ H_0 \cap H \mid H \in \CA'\}$.
Then $(\CA, \CA', \CA'')$ is a \emph{triple} of arrangements,
\cite[Def.\ 1.14]{orlikterao:arrangements}.

The \emph{product}
$\CA = (\CA_1 \times \CA_2, V_1 \oplus V_2)$ 
of two arrangements $(\CA_1, V_1), (\CA_2, V_2)$
is defined by
\begin{equation}
\label{eq:product}
\CA = \CA_1 \times \CA_2 := \{H_1 \oplus V_2 \mid H_1 \in \CA_1\} \cup 
\{V_1 \oplus H_2 \mid H_2 \in \CA_2\},
\end{equation}
see \cite[Def.\ 2.13]{orlikterao:arrangements}.
In particular,  
$|\CA| = |\CA_1| + |\CA_2|$. 

Let $\CA = \CA_1 \times \CA_2$ be a product. 
By  \cite[Prop.\ 2.14]{orlikterao:arrangements},
there is a lattice isomorphism
\begin{equation}
\label{eq:latticesum}
 L(\CA_1) \times L(\CA_2) \cong L(\CA) \quad \text{by} \quad
(X_1, X_2) \mapsto X_1 \oplus X_2.
\end{equation}
Using \eqref{eq:product}, it is easily seen that
for $X =  X_1 \oplus X_2 \in L(\CA)$, we have 
\begin{equation}
\label{eq:productAX}
\CA _X =  ({\CA_1})_{X_1} \times ({\CA_2})_{X_2}
\end{equation} and 
\begin{equation}
\label{eq:restrproduct}
\CA^X = \CA_1^{X_1} \times \CA_2^{X_2}.
\end{equation}

\subsection{Free  and inductively free Arrangements}
\label{ssect:free}

Free arrangements play a crucial role in the theory of arrangements;
see \cite[\S 4]{orlikterao:arrangements} for the definition and 
basic properties. If $\CA$ is free, then 
we can associate with $\CA$ the multiset of its \emph{exponents}, 
denoted $\exp \CA$. 

Owing to \cite[Prop.\ 4.28]{orlikterao:arrangements}, 
free arrangements behave well with respect to 
the  product construction.

\begin{proposition}
\label{prop:product-free}
Let $\CA_1, \CA_2$ be two arrangements.
Then  $\CA = \CA_1 \times \CA_2$ is free
if and only if both 
$\CA_1$ and $\CA_2$ are free and in that case
$\exp \CA = \{\exp \CA_1, \exp \CA_2\}$.
\end{proposition}

Terao's celebrated \emph{Addition-Deletion Theorem} 
\cite{terao:freeI} plays a 
fundamental role in the study of free arrangements, 
\cite[Thm.\ 4.51]{orlikterao:arrangements}.

\begin{theorem}
\label{thm:add-del}
Suppose that $\CA \ne \Phi_\ell$.
Let  $(\CA, \CA', \CA'')$ be a triple of arrangements. Then any 
two of the following statements imply the third:
\begin{itemize}
\item[(i)] $\CA$ is free with $\exp \CA = \{ b_1, \ldots , b_{\ell -1}, b_\ell\}$;
\item[(ii)] $\CA'$ is free with $\exp \CA' = \{ b_1, \ldots , b_{\ell -1}, b_\ell-1\}$;
\item[(iii)] $\CA''$ is free with $\exp \CA'' = \{ b_1, \ldots , b_{\ell -1}\}$.
\end{itemize}
\end{theorem}

Terao's \emph{Factorization Theorem}
\cite{terao:freefactors} shows
that the Poincar\'e polynomial 
of a free arrangement $\CA$
factors into linear terms 
given by the exponents of $\CA$
(cf.\ \cite[Thm.\ 4.137]{orlikterao:arrangements}):

\begin{theorem}
\label{thm:freefactors}
Suppose that 
$\CA$ is free with $\exp \CA = \{ b_1, \ldots , b_\ell\}$.
Then 
\[
\pi(\CA,t) = \prod_{i=1}^\ell (1 + b_i t).
\]
\end{theorem}

Theorem \ref{thm:add-del} motivates the notion of an 
\emph{inductively free} arrangement,   
\cite[Def.\ 4.53]{orlikterao:arrangements}.

\begin{defn}
\label{def:indfree}
The class $\CIF$ of \emph{inductively free} arrangements 
is the smallest class of arrangements subject to
\begin{itemize}
\item[(i)] $\Phi_\ell \in \CIF$ for each $\ell \ge 0$;
\item[(ii)] if there exists a hyperplane $H_0 \in \CA$ such that both
$\CA'$ and $\CA''$ belong to $\CIF$, and $\exp \CA '' \subseteq \exp \CA'$, 
then $\CA$ also belongs to $\CIF$.
\end{itemize}
\end{defn}

In \cite[Prop.\ 2.10]{hogeroehrle:indfree},
we showed that the compatibility 
of products with free arrangements
from Proposition \ref{prop:product-free}
restricts to inductively free arrangements.

\begin{proposition}
\label{prop:product-indfree}
Let $\CA_1, \CA_2$ be two arrangements.
Then  $\CA = \CA_1 \times \CA_2$ is inductively free
if and only if both 
$\CA_1$ and $\CA_2$ are inductively free and in that case
$\exp \CA = \{\exp \CA_1, \exp \CA_2\}$.
\end{proposition}

Following \cite[Def.\ 4.140]{orlikterao:arrangements}, 
we say that $\CA$ is \emph{hereditarily free}
provided $\CA^X$ is free for every $X \in L(\CA)$.
In general, a free arrangement need not be hereditarily free,
\cite[Ex.\ 4.141]{orlikterao:arrangements}.

\subsection{Supersolvable Arrangements}
\label{ssect:super}

Let $\CA$ be an arrangement.
Following \cite[\S 2]{orlikterao:arrangements}, we say
that $X \in L(\CA)$ is \emph{modular}
provided $X + Y \in L(\CA)$ for every $Y \in L(\CA)$, 
cf.\ \cite[Def.\ 2.32, Cor.\ 2.26]{orlikterao:arrangements}.
The following notion is due to Stanley \cite{stanley:super}. 

\begin{defn}
\label{def:super}
Let $\CA$ be a central arrangement of rank $r$.
We say that $\CA$ is \emph{supersolvable} 
provided there is a maximal chain
\[
V = X_0 < X_1 < \ldots < X_{r-1} < X_r = T_\CA
\]
of modular elements $X_i$ in $L(\CA)$.
\end{defn}

\begin{remark}
\label{rem:2-arr}
By \cite[Ex.\ 2.28]{orlikterao:arrangements}, 
$V$, $T_\CA$ and the members in $\CA$ 
are always modular in $L(\CA)$.
It follows that all $0$- $1$-, and $2$-arrangements are supersolvable.
\end{remark}

Note that supersolvable arrangements
are inductively free, 
\cite[Thm.\ 4.58]{orlikterao:arrangements}.

\subsection{Nice Arrangements}
\label{ssect:factored}

The notion of a \emph{nice} or \emph{factored} 
arrangement goes back to Terao \cite{terao:factored}.
It generalizes the concept of a supersolvable arrangement, see
Proposition \ref{prop:ssfactored}.
We recall the relevant notions and results from \cite{terao:factored} 
(cf.\  \cite[\S 2.3]{orlikterao:arrangements}).

\begin{defn}
\label{def:independent}
Let $\pi = (\pi_1, \ldots , \pi_s)$ be a partition of $\CA$.
Then $\pi$ is called \emph{independent}, provided 
for any choice $H_i \in \pi_i$ for $1 \le i \le s$,
the resulting $s$ hyperplanes are linearly independent, i.e.\
$r(H_1 \cap \ldots \cap H_s) = s$.
\end{defn}

\begin{defn}
\label{def:indpart}
Let $\pi = (\pi_1, \ldots , \pi_s)$ be a partition of $\CA$
and let $X \in L(\CA)$.
The \emph{induced partition} $\pi_X$ of $\CA_X$ is given by the non-empty 
blocks of the form $\pi_i \cap \CA_X$.
\end{defn}

\begin{defn}
\label{def:factored}
The partition 
$\pi$ of $\CA$ is
\emph{nice} for $\CA$ or a \emph{factorization} of $\CA$  provided 
\begin{itemize}
\item[(i)] $\pi$ is independent, and 
\item[(ii)] for each $X \in L(\CA) \setminus \{V\}$, the induced partition $\pi_X$ admits a block 
which is a singleton. 
\end{itemize}
If $\CA$ admits a factorization, then we also say that $\CA$ is \emph{factored} or \emph{nice}.
\end{defn}

\begin{remark}
\label{rem:factored}
(i). 
Vacuously, the empty partition is nice for the empty arrangement $\Phi_\ell$.

(ii).
If $\CA \ne \Phi_\ell$,  
$\pi$ is a nice partition of $\CA$
and $X \in L(\CA)\setminus\{V\}$, then the non-empty parts of the 
induced partition $\pi_X$ in turn form a nice partition of $\CA_X$;
cf.~the proof of \cite[Cor.\ 2.11]{terao:factored}

(iii). 
Since the singleton condition in Definition \ref{def:factored}(ii)
also applies to the center $T_\CA$ of $L(\CA)$, 
a factorization $\pi$ of $\CA \ne \Phi_\ell$
always admits a singleton as one of its parts.
Also note that for a hyperplane, 
the singleton condition trivially holds.

(iv).
Usually, when $\CA$ is factored, there is more than one 
nice partition. However, there are instances
when $\CA$ admits a unique nice partition, 
see Example \ref{ex:indfreefactored-notindfactored}.
\end{remark}

We recall the main results from \cite{terao:factored} 
(cf.\  \cite[\S 3.3]{orlikterao:arrangements}) that motivated 
Definition \ref{def:factored}.

\begin{theorem}
\label{thm:teraofactored}
Let $\CA$ be a central $\ell$-arrangement and let 
$\pi = (\pi_1, \ldots, \pi_s)$ be a partition of $\CA$.
Then the $\BBK$-linear map $\kappa$ 
defined in \eqref{eq:factoredOS}
is an isomorphism of graded $\BBK$-vector spaces
if and only if $\pi$ is nice for $\CA$.
\end{theorem}

\begin{corollary}
\label{cor:teraofactored}
Let  $\pi = (\pi_1, \ldots, \pi_s)$ be a factorization of $\CA$.
Then the following hold:
\begin{itemize}
\item[(i)] $s = r = r(\CA)$ and 
\[
\Poin(A(\CA),t) = \prod_{i=1}^r (1 + |\pi_i|t);
\]
\item[(ii)]
the multiset $\{|\pi_1|, \ldots, |\pi_r|\}$ only depends on $\CA$;
\item[(iii)]
for any $X \in L(\CA)$, we have
\[
r(X) = |\{ i \mid \pi_i \cap \CA_X \ne \varnothing \}|.
\]  
\end{itemize}
\end{corollary}

\begin{remark}
\label{rem:factorediscombinatorial}
It follows from \eqref{eq:piA}
and Corollary \ref{cor:teraofactored} that 
the question whether $\CA$ is factored is a purely combinatorial 
property and only depends on the lattice $L(\CA)$. 
\end{remark}

\begin{remark}
\label{rem:exponents}
Suppose that $\CA$ is free of rank $r$.
Then $\CA = \Phi_{\ell-r} \times \CA_0$, 
where $\CA_0$ is an essential, free $r$-arrangement
(cf.\ \cite[\S 3.2]{orlikterao:arrangements}), and so,  
thanks to Proposition \ref{prop:product-free}, 
$\exp \CA = \{0^{\ell-r}, \exp \CA_0\}$.
Suppose that $\pi = (\pi_1, \ldots, \pi_r)$ is a nice partition 
of $\CA$. 
Then by the factorization properties of the Poincar\'e polynomials
for free and factored arrangements, 
Theorem \ref{thm:freefactors}, respectively 
Corollary \ref{cor:teraofactored}(i) and 
\eqref{eq:piA}
we have 
\begin{equation}
\label{eq:exp}
\exp \CA = \{0^{\ell-r}, |\pi_1|, \ldots, |\pi_r|\}.
\end{equation}
In particular, if $\CA$ is essential, then 
\begin{equation}
\label{eq:ess-exp}
\exp \CA = \{|\pi_1|, \ldots, |\pi_\ell|\}.
\end{equation}
\end{remark}

Finally, we record 
\cite[Ex.\ 2.4]{terao:factored},
which shows that nice arrangements generalize 
supersolvable ones 
(cf.\  \cite[Thm.\ 5.3]{orliksolomonterao:hyperplanes}, 
\cite[Prop.\ 3.2.2]{jambu:factored}, 
\cite[Prop.\ 2.67, Thm.\ 3.81]{orlikterao:arrangements}).

\begin{proposition}
\label{prop:ssfactored}
Let $\CA$ be a central,  supersolvable arrangement of rank $r$.
Let 
\[
V = X_0 < X_1 < \ldots < X_{r-1} < X_r = T_\CA
\]
be a  maximal chain of modular elements in $L(\CA)$.
Define
$\pi_i = \CA_{X_i} \setminus \CA_{X_{i-1}}$
for $1 \le i \le r$.
Then $\pi = (\pi_1, \ldots, \pi_r)$ is a nice partition of $\CA$.
In particular, the $\BBK$-linear map 
$\kappa$ defined in \eqref{eq:factoredOS} is an 
isomorphism of graded $\BBK$-vector spaces.
\end{proposition}

\section{Factored and inductively factored Arrangements}
\label{sec:indfactored}

\subsection{Restriction, Addition and Deletion for nice Arrangements}
\label{sec:adddel}

In our main result, Theorem \ref{thm:add-del-factored}, 
we prove an analogue for nice arrangements 
of Terao's seminal
Addition-Deletion Theorem \ref{thm:add-del} for free arrangements.
Following Jambu and Paris 
\cite{jambuparis:factored}, 
we introduce further notation.

\begin{defn} 
\label{def:distinguished}
Suppose $\CA \ne \Phi_\ell$.
Let $\pi = (\pi_1, \ldots, \pi_s)$ be a partition of $\CA$.
Let $H_0 \in \pi_1$ and 
let $(\CA, \CA', \CA'')$ be the triple associated with $H_0$. 
We say that $H_0$ is \emph{distinguished (with respect to $\pi$)}
provided $\pi$ induces a factorization $\pi'$ of 
$\CA'$, i.e.\ the non-empty 
subsets $\pi_i \cap \CA'$ form a nice partition
of $\CA'$. Note that since $H_0 \in \pi_1$, we have
$\pi_i \cap \CA' = \pi_i\not= \varnothing$ 
for $i = 2, \ldots, s$. 

Also, associated with $\pi$ and $H_0$, we define 
the \emph{restriction map}
\[
\R := \R_{\pi,H_0} : \CA \setminus \pi_1 \to \CA''\ \text{ given by } \ H \mapsto H \cap H_0
\]
and set 
\[
\pi_i'' := \R(\pi_i) = \{H \cap H_0 \mid H \in \pi_i\} \
\text{ for }\  2 \le i \le s.
\]

In general, $\R$ need not be surjective nor injective.
However, since we are only concerned with cases when 
$\pi'' = (\pi_2'', \ldots, \pi_s'')$ is a
partition of $\CA''$,  
$\R$ has to be onto and 
$\R(\pi_i) \cap \R(\pi_j) = \varnothing$ for $i \ne j$.
As we shall see, 
the natural condition for us is the injectivity of $\R$.
\end{defn}

\begin{remark}
\label{rem:distinguished}
Our definitions of a distinguished hyperplane 
and of the restriction map $\R$
in Definition \ref{def:distinguished} differ 
from the one by Jambu and Paris 
\cite[\S 2]{jambuparis:factored},
in that we do not require the underlying partition to 
be a factorization of $\CA$.
This more general setting is crucial for the purpose of 
the ``Addition'' statement in 
Theorem \ref{thm:add-del-factored} below, as here 
we want to deduce that $\pi$ is nice for $\CA$. 
This comes at the expense of having to impose  
bijectivity for $\R$. 

These more general notions are clearly feasible, 
as shown by the following example.
\end{remark}

\begin{example}
\label{ex:a222}
Let $\CA = \CD_2^2$ be the reflection arrangement
of the Coxeter group of type $B_2$,
cf.\ \cite[Ex.~2.6]{jambuterao:free}. 
Then $\CA$ has defining polynomial $Q(\CA) = xy(x^2-y^2)$.
While independent,  
the  partition
$\pi = (\{\ker x, \ker(x-y)\}, \{\ker y, \ker(x+y)\})$ 
is obviously
not nice for $\CA$, as none of its  
parts is a singleton, cf.\ Remark \ref{rem:factored}(iii). 
However, for any choice of hyperplane $H_0 \in \CA$, the induced partitions
$\pi'$ of $\CA'$ and $\pi''$ of $\CA''$  are factorizations, 
but $\R$ is never injective.
Nevertheless,  as a $2$-arrangment, $\CA$ is of course factored,
cf.\ Remark \ref{rem:dim2}.
\end{example}

Using the notation and terminology introduced 
in Definition \ref{def:distinguished},
the following result by Jambu and Paris
gives the ``\emph{Restriction}'' part of the Addition-Deletion Theorem
\ref{thm:add-del-factored} below. 

\begin{proposition}
[{\cite[Prop.\ 2.1]{jambuparis:factored}}] 
\label{prop:JPinduced}
Suppose that $\CA \ne \Phi_\ell$ has rank $r$.
Let $\pi = (\pi_1, \ldots, \pi_r)$ be a factorization of $\CA$.
Let $H_0 \in \pi_1$ and 
let $(\CA, \CA', \CA'')$ be the triple associated with $H_0$. 
Suppose $H_0$ is distinguished with respect to $\pi$.
Then the restriction map $\R: \CA \setminus \pi_1 \to \CA''$ 
is a bijection and
the induced partition $\pi'' = (\pi_2'', \ldots, \pi_r'')$ is a 
factorization  of $\CA''$.
\end{proposition}

\begin{proof}
Let $X \in \CA''$. 
Viewing $X$ as a member of $L(\CA)$, we have $r(X) = 2$. 
It follows from Corollary \ref{cor:teraofactored}(iii) that
there is an $H \in \pi_k \cap \CA_X$ for some $k >1$.
Thus $\R(H) = X$, and so $\R$ is onto.
Suppose $H, H' \in \CA\setminus \pi_1$ so that 
$H_0 \cap H = X = H_0 \cap H'$.
Again since $r(X) = 2$ and $H_0 \in \pi_1 \cap \CA_X$, we have
$H, H' \in \pi_k$ for some $k >1$,
by Corollary \ref{cor:teraofactored}(iii).
In particular, $H \cap H'$ belongs to $L(\CA')$.
Since $\pi'$ is a factorization of $\CA'$, 
$\pi'_{H \cap H'}$ has to admit a singleton as one of its parts,
say $\pi_j' \cap \CA_{H \cap H'}' = \{K\}$.
If $j = 1$, then $K \ne H_0$ and so $|\pi_1 \cap \CA_X| \ge 2$.
Consequently, the singleton of 
$\pi_X$ has to be $\pi_k \cap \CA_X$ and thus $H = H'$.
If $j > 1$, then 
$\pi_j' \cap \CA_{H \cap H'}' = \pi_j \cap \CA_{H \cap H'}$, 
and so $j = k$ and again $H = H'$.
In any event, $\R$ is bijective.

Let $H_i \in \pi_i$ for $i = 2, \ldots, r$.
Then $r(H_0 \cap H_2 \cap \ldots \cap H_r) = r$, since $\pi$ is
independent. Thus 
$r((H_0 \cap H_2) \cap \ldots \cap (H_0 \cap H_r)) = r - 1$ 
in $L(\CA'')$, and so $\pi''$ is independent.

Let $X \in L(\CA'')$ and $X \ne H_0$.
Viewing $X$ as a member of $L(\CA)$,  
since $\pi$ is a factorization, there is a part $\pi_k$ of $\pi$
so that $\pi_k \cap \CA_X = \{H\}$.
If $k > 1$, then $H \in \CA \setminus \pi_1$ and so
$\pi_k'' \cap \CA''_X = \{H \cap H_0\}$. %, as $\R$ is bijective.

Now suppose that $k = 1$. Then, since $H_0 \in \pi_1 \cap \CA_X$, 
we must have $H = H_0$, i.e.\ $\pi_1 \cap \CA_X = \{H_0\}$.
Write $X = H_0 \cap Y$, where $Y = L_1 \cap \ldots \cap L_m$
for $L_i \in \CA \setminus \pi_1$.

If $X = Y$, then $X \in L(\CA')$.
Thus, since $\pi'$ is a factorization of $\CA'$, there is a 
part $\pi_j'$ of $\pi'$ so that 
$\pi_j' \cap \CA_X' = \{K\}$. Since $K \ne H_0$ and 
$\pi_1 \cap \CA_X = \{H_0\}$, we must have $j > 1$ 
and so $K \in \CA \setminus \pi_1$.
Therefore, $\pi_j \cap \CA_X = \pi_j' \cap \CA_X' = \{K\}$ and so 
$\pi_j'' \cap \CA_X'' = \{K \cap H_0\}$.
Finally, if $X \ne Y$, then $r(X) = r(Y) +1$ and thus
$\pi_1 \cap \CA_Y = \varnothing$, thanks to 
Corollary \ref{cor:teraofactored}(iii).
Thus, since $\pi_1 \cap \CA_X = \{H_0\}$, we have $\CA_X = \CA_Y \cup \{H_0\}$.
As $\pi$ is a factorization, there is an index $j >1$ such that
$\pi_j \cap \CA_Y = \{L\}$. 
Since $\CA_X = \CA_Y \cup \{H_0\}$ and $j >1$, we have
\[
\{ L \} = \pi_j \cap \CA_Y = \pi_j \cap \CA_X. 
\]
It follows that
$\pi_j'' \cap \CA_X'' = \{L \cap H_0\}$,
as required.
\end{proof}

Here is our analogue for nice arrangements of Terao's 
Addition-Deletion Theorem \ref{thm:add-del} for free arrangements.

\begin{theorem}
\label{thm:add-del-factored}
Suppose that $\CA \ne \Phi_\ell$ has rank $r$.
Let $\pi = (\pi_1, \ldots, \pi_s)$ be a  partition  of $\CA$.
Let $H_0 \in \pi_1$ and 
let $(\CA, \CA', \CA'')$ be the triple associated with $H_0$. 
Then any two of the following statements imply the third:
\begin{itemize}
\item[(i)] $\pi$ is nice for $\CA$;
\item[(ii)] $\pi'$ is nice for $\CA'$;
\item[(iii)] $\R: \CA \setminus \pi_1 \to \CA''$ 
is bijective and $\pi''$ is nice for $\CA''$.
\end{itemize}
\end{theorem}

\begin{proof}
If (i) and (ii) hold, then so does (iii), by Proposition \ref{prop:JPinduced}.

For ``\emph{Addition}'', assume (ii) and (iii).
We need to show that $\pi$
is a factorization of $\CA$.
As remarked in Definition \ref{def:distinguished}, 
if $\R$ is injective and $\pi''$ is a partition
of $\CA''$, then $\R$ is bijective.

By (ii), $\pi'$ is already independent.
Let $H_i \in \pi_i$ for  $i = 2, \ldots, s$.
Set $X = H_0 \cap H_2 \cap \ldots \cap H_s$.
Then 
$X = (H_0 \cap H_2) \cap (H_0 \cap H_3) \cap \ldots \cap (H_0 \cap  H_s) \in L(\CA'')$.
Since $H_i \in \pi_i$, we have
$H_0 \cap H_i \in \pi''_i$, for  $i = 2, \ldots, s$.
By (iii), $\pi''$ is a factorization of $\CA''$, so 
the rank of $X$ as a member of $L(\CA'')$ is $r - 1$, 
by Corollary \ref{cor:teraofactored}(i) applied to $\CA''$,  
and thus 
the rank of $X$ in $L(\CA)$ is $r$, 
and so in particular, $s = r$ and
$\pi$ is independent.

Let $X \in L(\CA)\setminus \{V\}$. 
We need to show that $\pi_X$ admits a singleton as 
one of its parts.
If $X \in \CA$ is a hyperplane, this is obvious.
So we may assume that $r(X) >1$.

If $H_0 \notin \CA_X$, 
then $\CA_X = \CA'_X$ and 
$X \in L(\CA') \setminus \{V\}$. 
Since $\pi'$ is a factorization of $\CA'$,
there is an index $k$ such that 
$\pi_k' \cap \CA'_X = \{H\}$.
Therefore, since $H \ne H_0$ and 
$\CA_X = \CA'_X$, 
\[
\pi_k \cap \CA_X = \pi_k \cap \CA'_X = \pi_k' \cap \CA'_X = \{H\}
\]
is a singleton.

Now suppose that $H_0 \in \CA_X$.
Since $\R$ is surjective, there are $L_1,\ldots,L_m \in \CA \setminus \pi_1$ 
with $\CA''_X = \{H_0 \cap L_1,\ldots,H_0 \cap L_m\}$.
Since $\pi''$ is a factorization of $\CA''$, there is a $k$ so that
$\pi_k'' \cap \CA''_X = \{H_0 \cap L_i\}$ for some $i$.
Since $\R$ is injective it follows that $\pi_k \cap \CA_X = \{L_i\}$
is a singleton.
Consequently, $\pi$ is a factorization of $\CA$, as claimed.

Finally, 
for the ``\emph{Deletion}'' part, 
suppose that (i) and (iii) hold. 
We need to show that $\pi'$ is a factorization
of $\CA'$.
Clearly, $\pi'$ is independent, since $\pi$ is.

Let $X \in L(\CA')\setminus \{V\}$.
We may also assume that $X \notin \CA'$.
Then $X \in L(\CA)\setminus \{V\}$.
Since $\pi$ is a factorization of $\CA$,
there is an index $k$ so that
$\pi_k \cap \CA_X =\{H\}$
is a singleton.
If $H_0 \notin \CA_X$,
then $H \ne H_0$ and so $H \in \CA'_X$.
Therefore, 
\[
\{H \} \subseteq \pi_k' \cap \CA'_X \subseteq \pi_k \cap \CA_X =  \{H\}.
\]
So $\pi'_k \cap \CA'_X$ is a singleton.

Now suppose that $H_0 \in \CA_X$, so that $X \subset H_0$. 
Then again 
$X = H_0 \cap K_1 \cap \ldots \cap K_m$
for some hyperplanes $K_i \in \CA$.
Thus $X = (H_0 \cap K_1) \cap (H_0 \cap K_2) \cap \ldots \cap (H_0 \cap  K_m) \in L(\CA'')$.
Because $\R : \CA \setminus \pi_1 \to \CA''$ is surjective, 
there are hyperplanes $L_i \in \CA \setminus \pi_1$ 
such that $H_0 \cap L_i = H_0 \cap K_i$ for 
$i = 1, \ldots, m$.
So $X = (H_0 \cap L_1) \cap (H_0 \cap L_2) \cap \ldots \cap (H_0 \cap  L_m)$.
Since $\pi''$ is a factorization of $\CA''$,
there is a $k$ so that   
$\pi_k'' \cap \CA''_X = \{H_0 \cap L_i\}$.
By construction, $k > 1$,
$L_i \in  \pi_k = \pi_k'$ and $L_i \in \CA'$.
Thus, since $\R$ is injective,
it follows that 
\[
\{L_i \} = \pi_k \cap \CA_X = \pi'_k \cap \CA'_X.
\]
So again, $\pi'_k \cap \CA'_X$ is a singleton.
Consequently, $\pi'$ is a factorization of $\CA'$, as claimed.
\end{proof}

Theorem \ref{theorem:add-del-factored} is an immediate 
consequence of Theorem \ref{thm:add-del-factored}.

As nice arrangements need not be free and 
vice versa, combining Theorems 
\ref{thm:add-del} and 
\ref{thm:add-del-factored}
yields an 
Addition-Deletion Theorem for 
the subclass of arrangements that 
are both nice and free.

\begin{theorem}
\label{thm:add-del-free-factored}
Suppose that $\CA \ne \Phi_\ell$ has rank $r$.
Let $\pi = (\pi_1, \ldots, \pi_s)$ be a  partition  of $\CA$.
Let $H_0 \in \pi_1$ and 
let $(\CA, \CA', \CA'')$ be the triple associated with $H_0$. 
Then any two of the following statements imply the third:
\begin{itemize}
\item[(i)] $\pi$ is nice for $\CA$ and $\CA$ is free; 
\item[(ii)] $\pi'$ is nice for $\CA'$ and $\CA'$ is free; 
\item[(iii)] $\pi''$ is nice for $\CA''$ and $\CA''$ is free with
$\exp \CA'' = \{0^{\ell-r}, |\pi_2|, \ldots, |\pi_r|\}$.
\end{itemize}
\end{theorem}

\begin{proof}
By \eqref{eq:exp}, we have 
$\exp \CA = \{0^{\ell-r}, |\pi_1|, \ldots, |\pi_r|\}$, 
$\exp \CA' = \{0^{\ell-r}, |\pi_1|-1, \ldots, |\pi_r|\}$, and 
$\exp \CA'' = \{0^{\ell-r}, |\pi_2''|, \ldots, |\pi_r''|\}$,
in (i), (ii) and (iii), respectively.
In addition, in (iii),  
since $|\pi_i| \ge |\R(\pi_i)| = |\pi_i''|$ and 
$\exp \CA'' = \{0^{\ell-r}, |\pi_2|, \ldots, |\pi_r|\}$, it follows that 
$|\pi_i| = |\pi_i''|$, for each $i = 2, \ldots, r$
and so $\R: \CA \setminus \pi_1 \to \CA''$ is bijective.
The result now follows from 
Theorems \ref{thm:add-del} and 
\ref{thm:add-del-factored}.
\end{proof}

We obtain Theorem \ref{theorem:add-del-free-factored2} as a variation of 
Theorem \ref{thm:add-del-free-factored};
the requirements on the exponents 
needed for Theorem \ref{thm:add-del} 
follow from Remark \ref{rem:exponents}.

\begin{remark}
\label{rem:injectivity}
The injectivity condition on $\R$
in Theorem \ref{thm:add-del-factored}(iii)
and the requirement on the exponents of $\CA''$ in
Theorem \ref{thm:add-del-free-factored}(iii)
(likewise the injectivity condition on $\R$ in 
Theorems \ref{theorem:add-del-factored}
and \ref{theorem:add-del-free-factored2}) 
are necessary. Else both the ``Addition'' and ``Deletion''
statements in each of the theorems are wrong,  
see Examples \ref{ex:a222} and \ref{ex:g333},
respectively.
\end{remark}

\subsection{Inductively factored Arrangements}
\label{ssec:indfactoredindfree}

The Addition-Deletion Theorem \ref{thm:add-del-factored} 
for nice arrangements motivates
the following stronger notion of factorization, 
cf.\ \cite{jambuparis:factored}.

\begin{defn} 
\label{def:indfactored}
The class $\CIFAC$ of \emph{inductively factored} arrangements 
is the smallest class of pairs $(\CA, \pi)$ of 
arrangements $\CA$ together with a partition $\pi$
subject to
\begin{itemize}
\item[(i)] $(\Phi_\ell, (\varnothing)) \in \CIFAC$ for each $\ell \ge 0$;
\item[(ii)] if there exists a partition $\pi$ of $\CA$ 
and a hyperplane $H_0 \in \pi_1$ such that 
for the triple $(\CA, \CA', \CA'')$ associated with $H_0$ 
the restriction map $\R = \R_{\pi, H_0} : \CA \setminus \pi_1 \to \CA''$ 
is bijective and for the induced partitions $\pi'$ of $\CA'$ and 
$\pi''$ of $\CA''$ 
both $(\CA', \pi')$ and $(\CA'', \pi'')$ belong to $\CIFAC$, 
then $(\CA, \pi)$ also belongs to $\CIFAC$.
\end{itemize}
If $(\CA, \pi)$ is in $\CIFAC$, then we say that
$\CA$ is \emph{inductively factored with respect to $\pi$}, or else
that $\pi$ is an \emph{inductive factorization} of $\CA$. 
Sometimes, we simply say $\CA$ is \emph{inductively factored} without 
reference to a specific inductive factorization of $\CA$.
\end{defn}

\begin{remark}
Our Definition \ref{def:indfactored}
of inductively factored arrangements
differs from the one given by 
Jambu and Paris \cite{jambuparis:factored}
in that, apart from the mere technicalities of 
incorporating empty arrangements and for defining
$\CIFAC$ for pairs of arrangements and partitions 
rather than for partitions of arrangements, 
in part (ii) we do not assume from the outset 
that $\pi$ is a factorization of $\CA$.
This is possible by virtue of 
the Addition-Deletion Theorem \ref{thm:add-del-factored}.
However, this comes at the cost of 
the bijectivity requirement for the associated restriction map $\R$;
cf.\ Remark \ref{rem:injectivity}.
\end{remark}

In view of Proposition \ref{prop:JPinduced}
it is desirable to have an easy condition that 
ensures the existence of a distinguished
hyperplane with respect to 
a given factorization $\pi$ of $\CA$. 
Our next observation shows that a modular element in $L(\CA)$ 
of rank $r -1$
which is compatible with a given factorization $\pi$ of $\CA$ 
gives such a condition. 
While sufficient, the presence 
of such a modular element is not necessary for the 
existence of a  distinguished hyperplane, see
Example \ref{ex:d13}.

\begin{lemma}
\label{lem:distinguished}
Let $\pi = (\pi_1,\ldots, \pi_r)$ be a factorization of $\CA$.
Suppose $Z \in L(\CA)$ is modular of rank $r -1$ 
so that $\pi_1 = \CA \setminus \CA_Z$. 
Then any $H \in \pi_1$ is distinguished with respect to $\pi$.
In particular, 
$\R : \CA_Z \to \CA''$ is bijective.
\end{lemma}

\begin{proof}
Let $H_0 \in \pi_1$ and let 
$(\CA, \CA', \CA'')$ be the triple with respect to $H_0$.
We have to show that $\pi'$ is a factorization of $\CA'$.
Since $\pi$ is independent, clearly so is $\pi'$. 

Let $X \in L(\CA')\setminus \{V\}$. Then $X$ also belongs to 
$L(\CA)\setminus \{V\}$. 
Since $Z$ is modular, $X + Z \in L(\CA)$.
First suppose that $X + Z = V$.
Then $r(X \cap Z) = r$, as $Z \not\subseteq X$.
Since $Z$ is modular, it follows from 
\cite[Lem.\ 2.24]{orlikterao:arrangements}
that $1 \le r(X) = r(X \cap Z) + r(X + Z) - r(Z) = r + 0 - (r-1)$. 
Thus $r(X) = 1$ and so $X$ is a 
hyperplane in $\CA'$. 
In particular, $X \ne H_0$ and $\CA'_X =\{X\}$. 
Thus there is a part $\pi'_k$ of $\pi'$ so that
$\pi'_k \cap \CA'_X =\{X\}$ is a singleton.

Now suppose that  $X + Z \ne V$.
Since $X + Z$ belongs to $ L(\CA)$,
there is a block $\pi_k$ of $\pi$ such that
$\pi_k \cap \CA_{X+Z}$ is a singleton.
Because $\CA_{X+Z} = \CA_X \cap \CA_Z$ and $\pi_1 \cap \CA_Z = \varnothing$,
we must have $k > 1$ and so, 
$H_0 \notin \pi_k =  \pi_k'$ and $\pi_k = \pi_k \cap \CA_Z$.
It follows that
\[
\pi_k' \cap \CA'_X = \pi_k \cap \CA_X  = \pi_k \cap \CA_Z \cap \CA_X = 
\pi_k \cap \CA_{X+Z} 
\]
 is a singleton, as required.
The final statement on $\R$ follows from Proposition \ref{prop:JPinduced}.
\end{proof}

Note, the presence of the modular element $Z$ 
in Lemma \ref{lem:distinguished} implies that the 
bijective map $\R : \CA_Z \to \CA''$ induces an isomorphism of
the corresponding Orlik-Solomon algebras: 
$A(\CA_Z) \cong  A(\CA'')$,
cf.\ \cite[Lem.\ 5.1]{orliksolomonterao:hyperplanes},
\cite[Lem.\ 3.79]{orlikterao:arrangements}.

As an application of Lemma \ref{lem:distinguished}, we can 
strengthen Proposition \ref{prop:ssfactored}
(this is already stated in \cite{jambuparis:factored} without proof).
Note, the converse of 
Proposition \ref{prop:ssindfactored}
is false, see Example \ref{ex:d13}.

\begin{proposition}  
\label{prop:ssindfactored}
If $\CA$ is 
supersolvable, then $\CA$ is inductively factored.
\end{proposition}

\begin{proof}
Suppose that $\CA$ is supersolvable and 
let $V = X_0 < X_1 < \ldots < X_{r-1} < X_r = T_\CA$
be a  maximal chain of modular elements $X_i$ in $L(\CA)$.
Define $\pi_i = \CA_{X_i} \setminus \CA_{X_{i-1}}$
for $1 \le i \le r$. Then
$\pi = (\pi_1, \ldots, \pi_r)$ is nice for $\CA$, 
by Proposition \ref{prop:ssfactored}. 

We prove that $\pi$ is an inductive factorization of $\CA$ 
by induction on $n = |\CA|$. 
If $\CA = \Phi_\ell$, there is nothing to show.
So suppose that $n > 0$
and that the result holds for 
all supersolvable arrangements with less than $n$ hyperplanes.  

Let $H_0 \in \CA$ be a complement of $X_{r-1}$ in $V$ and let 
$(\CA, \CA', \CA'')$ be the triple with respect to $H_0$.
It follows from Lemma \ref{lem:distinguished} 
and Proposition \ref{prop:JPinduced}
that $\pi'$ is nice for $\CA'$, 
$\R : \CA \setminus\pi_r \to \CA''$ is bijective,
and $\pi''$ is nice for $\CA''$.

Thanks to 
\cite[Lem.\ 2.62(1)]{orlikterao:arrangements}, 
$\CA''$ is supersolvable and 
$\pi''$ is the nice partition afforded by the 
maximal chain of modular elements in $L(\CA'')$ 
induced by the given maximal chain in $L(\CA)$ above.
Since $|\CA''| < n$, it follows from our 
induction hypothesis that $\pi''$ is an inductive 
factorization of $\CA''$.
Likewise, by 
\cite[Lem.\ 2.62(2)]{orlikterao:arrangements}, 
$\CA'$ is supersolvable and 
$\pi'$ is the nice partition stemming from the 
maximal chain of modular elements in $L(\CA')$ 
induced by the given maximal chain in $L(\CA)$ above.
(Note that it might be the case that $r(\CA') = r(\CA) - 1$.)
Since $|\CA'| < n$, it follows from our 
induction hypothesis that $\pi'$ is an inductive 
factorization of $\CA'$.
Thus, by Definition \ref{def:indfactored}, 
$\pi$ is an inductive factorization of $\CA$, as desired.
\end{proof}

\begin{remark}
\label{rem:dim2}
Since any $1$- and $2$-arrangement is supersolvable,
by Remark \ref{rem:2-arr}, 
each such is inductively factored, by 
Proposition \ref{prop:ssindfactored}.
\end{remark}

In contrast to $2$-arrangements, 
a $3$-arrangement need not be factored.
It is easily seen that a $3$-arrangement
with at most $3$ hyperplanes is always
inductively factored.
However, already with $4$ hyperplanes
it need not even be factored, 
as our next example shows.

\begin{example}
\label{ex:notfactored}
One easily checks that 
the $3$-arrangement $\CA$ with defining polynomial
$Q(\CA) = xyz(x+y-z)$
does not admit a nice 
partition; neither is it free, see \cite[Ex.\ 4.34]{orlikterao:arrangements}.
This is the smallest such example. 
Simply by 
adding only one additional hyperplane
we obtain an inductively factored arrangement,
see Example \ref{ex:ot454}.
\end{example}

The similarity of 
Definitions \ref{def:indfree} and \ref{def:indfactored}  
is not a coincidence. Indeed, Jambu and Paris showed that 
inductively factored arrangements are inductively free
\cite[Prop.\ 2.2]{jambuparis:factored}.
(Jambu and Paris only claimed freeness but their 
proof actually does give the stronger result.)

\begin{proposition}
\label{prop:indfactoredindfree}
Let $\pi = (\pi_1, \ldots, \pi_r)$ be an inductive factorization of $\CA$. 
Then $\CA$ is inductively free with exponents 
$\exp \CA = \{0^{\ell-r}, |\pi_1|, \ldots, |\pi_r|\}$.
\end{proposition}

\begin{proof}
We argue by induction on $|\CA|$. 
If $|\CA| = 0$, there is nothing to show.
So suppose that $|\CA| = n > 0$ and that the result holds for 
all inductively factored arrangements with less than $n$ hyperplanes.  
Let $H_0 \in \pi_1$ be a distinguished hyperplane with respect to $\pi$. 
Then, by Definition \ref{def:indfactored}, 
$\pi'$ and $\pi''$ are inductive factorizations of $\CA'$ and $\CA''$.
As $|\CA'|, |\CA''| < n$, 
it follows from our induction hypothesis 
that both $\CA'$ and $\CA''$ are inductively free with 
exponents 
$\exp \CA' = \{0^{\ell-r}, |\pi_1|-1, |\pi_2|, \ldots, |\pi_r|\}$
and 
$\exp \CA'' = \{0^{\ell-r}, |\pi_2''|, \ldots, |\pi_r''|\}$.
Thanks to Proposition \ref{prop:JPinduced},
$\R : \CA \setminus \pi_1 \to \CA''$ is bijective, 
and so  
$\exp \CA'' 
%= \{0^{\ell-r}, |\pi_2''|, \ldots, |\pi_r''|\}
= \{0^{\ell-r}, |\pi_2|, \ldots, |\pi_r|\}$.
Consequently, 
$\exp \CA''  \subset \exp \CA'$, and so
$\CA$ is inductively free.
Finally, $\exp \CA = \{0^{\ell-r}, |\pi_1|, \ldots, |\pi_r|\}$,
by Theorem \ref{thm:add-del}.
\end{proof}

\begin{remark}
\label{rem:indfactored}
(i). 
Owing to 
Proposition \ref{prop:indfactoredindfree} and 
Remark \ref{rem:factorediscombinatorial},
the question whether $\CA$ is inductively factored is a combinatorial 
property and only depends on the lattice $L(\CA)$. 

(ii).
The converse of Proposition \ref{prop:indfactoredindfree} is false.
Terao has already noted that the reflection arrangement $\CA(D_4)$  of
the Coxeter group of type $D_4$ is not factored.
But $\CA(D_4)$ is inductively free, 
\cite[Ex.~2.6]{jambuterao:free}.

(iii).
It follows from Proposition \ref{prop:indfactoredindfree}
that Ziegler's example \cite[Ex.\ 4.1]{ziegler}
of a factored $3$-arrangement in 
characteristic $3$ is not inductively factored, as it is not free.

(iv).
An inductively free and factored arrangement need not be
inductively factored, see Example \ref{ex:indfreefactored-notindfactored}.
\end{remark}

\begin{remark}
\label{rem:indtable}
(i). 
If $\CA$ is inductively factored, then  $\CA$ is inductively free,
by Proposition \ref{prop:indfactoredindfree}.
The latter can be described by a so called 
\emph{induction table}, cf.~\cite[\S 4.3, p.~119]{orlikterao:arrangements}.
In this process we start with an inductively free arrangement
(frequently $\Phi_\ell$) and add hyperplanes successively ensuring that 
part (ii) of Definition \ref{def:indfree} is satisfied.
This process is referred to as \emph{induction of hyperplanes}.
This procedure amounts to 
choosing a total order on $\CA$, say 
$\CA = \{H_1, \ldots, H_n\}$, 
so that each of the subarrangements 
$\CA_0 := \Phi_\ell$, $\CA_i := \{H_1, \ldots, H_i\}$
and each of the restrictions $\CA_i^{H_i}$ is inductively free
for $i = 1, \ldots, n$.
In the associated induction table we record in the $i$-th row the information 
of the $i$-th step of this process, by 
listing $\exp \CA_i' = \exp \CA_{i-1}$, 
the defining form $\alpha_{H_i}$ of $H_i$, 
as well as $\exp \CA_i'' = \exp \CA_i^{H_i}$, 
for $i = 1, \ldots, n$.

The proof of Proposition \ref{prop:indfactoredindfree} shows 
that if $\pi$ is an inductive factorization of $\CA$
and $H_0 \in \CA$ is distinguished with respect to $\pi$, then 
the triple $(\CA, \CA', \CA'')$ with respect to $H_0$ 
is a triple of inductively free arrangements.
Thus an induction table of $\CA$ can be constructed,
compatible with suitable inductive factorizations of 
the subarrangments $\CA_i$.

(ii).
Now suppose $\CA$ is inductively free and let 
$\CA = \{H_1, \ldots, H_n\}$ be a 
choice of a total order on $\CA$, 
so that each of the subarrangements 
$\CA_0 := \Phi_\ell$, $\CA_i := \{H_1, \ldots, H_i\}$
and each of the restrictions $\CA_i^{H_i}$ is inductively free
for $i = 1, \ldots, n$.

Then, starting with the empty partition for $\Phi_\ell$, we can attempt to
build inductive factorizations $\pi_i$ of $\CA_i$
consecutively, resulting in an inductive factorization 
$\pi = \pi_n$ of $\CA = \CA_n$.
This is achieved by invoking Theorem \ref{thm:add-del-free-factored}
repeatedly in order to derive that each $\pi_i$ is an inductive factorization of $\CA_i$.
For this it suffices to check
the conditions in part (iii) of Theorem \ref{thm:add-del-free-factored}, i.e.,  that 
$\exp \CA_i''$ is given by the sizes of the parts of $\pi_i$ not containing $H_i$
and that the induced partition $\pi_i''$ of $\CA_i''$ is a factorization.
The fact that $H_i$ is distinguished
with respect to $\pi_i$ is part of the inductive hypothesis, as
$\pi_i' = \pi_{i-1}$ is an inductive factorization of $\CA_i' = \CA_{i-1}$.

We then add the inductive factorizations $\pi_i$
of $\CA_i$ as additional data into an induction table for $\CA$
(or else record to which part of $\pi_{i-1}$ the new hyperplane $H_i$ is appended to).
The data in such an extended induction table 
together with the ``Addition'' part of Theorem \ref{thm:add-del-free-factored}
then proves that $\CA$ is 
inductively factored. 
We refer to this technique as \emph{induction of factorizations} and the
corresponding table as an \emph{induction table of factorizations} for $\CA$.
Note that listing $\exp \CA_i'$ and $\exp \CA_i''$ in these tables is 
redundant, as the non-zero exponents are simply given by the cardinalities of the parts of
$\pi_i'$ and $\pi_i''$, respectively, cf.\ \eqref{eq:exp}.

We illustrate this induction of factorizations procedure in Tables \ref{table0}, \ref{table1} 
and \ref{indtable:non-HIFAC}
in Examples \ref{ex:ot454}, \ref{ex:d13} and \ref{ex:notheredfactored}, respectively.
As for the usual induction process for free arrangements, 
induction for factorizations is sensitive to the chosen order on $\CA$, 
cf.\ Example \ref{ex:ot454}.

It is worth noting that when using this inductive technique, we end up showing
that $\CA$ is inductively factored without knowing 
\emph{a priori} that $\CA$ is factored. 

Clearly, if $\CA$ is inductively free but not factored, then this 
induction of factorizations must terminate at a proper
subarrangement of  $\CA$, cf.\ Remark \ref{rem:indfactored}(ii).
\end{remark}

\begin{remark}
\label{rem:inducednotindfree}
Suppose $\CA$ and $H_0 \in \CA$ are such that  
$\CA$ is inductively factored but the deletion $\CA'$ is not.
Suppose that $\pi$ is a partition of $\CA$ such that 
$\pi'$ is nice  for $\CA'$ and that the conditions in  
Theorem \ref{thm:add-del-factored}(iii)  are satisfied.
Then, by Theorem \ref{thm:add-del-factored}, 
$\pi$ is nice for $\CA$, but it need not be 
the case that $\pi$ is  
an inductive factorization of $\CA$.
We present such an instance in the second part of Example \ref{ex:g333}.
\end{remark}

\subsection{Applications and Examples}
\label{subsec:examples}

In our first example we illustrate how to construct  
an induction table of factorizations as outlined in
Remark \ref{rem:indtable}(ii).
This example also shows that the order of hyperplanes
matters.

\begin{example}
\label{ex:ot454}
Let $\CA$ be the $3$-arrangement with defining polynomial  
\[
Q(\CA) = x y z (x + y) (x + y - z).
\] 
It follows from \cite[Ex.\ 4.54]{orlikterao:arrangements}
that $\CA$ is inductively free with $\exp \CA = \{1,2,2\}$. 
For simplicity, we enumerate the five 
hyperplanes of $\CA$ in the order their linear forms appear as factors in $Q(\CA)$, 
i.e.\ $H_1 = \ker(x), H_2 = \ker(y), \ldots, H_5 = \ker(x + y - z)$.
We claim that 
\[
\pi = (\pi_1, \pi_2, \pi_3) = (\{H_1\}, \{H_2, H_4\}, \{H_3, H_5\})
\]
is an inductive factorization of $\CA$.
Using the induction table of $\CA$ from  
\cite[Table 4.1]{orlikterao:arrangements}, 
we show this 
using the Addition-Deletion Theorem \ref{thm:add-del-free-factored}
by exhibiting an induction table of factorizations for $\CA$
in Table \ref{table0}. 
We indicate the corresponding inductive factorizations of $\CA_i'$ and $\CA_i''$
in the columns headed by $\pi_i'$ and $\pi_i''$, respectively. 
In this case it is particularly easy to verify 
the conditions from Theorem \ref{thm:add-del-free-factored}(iii): i.e.,  
$\exp \CA_i''$ is the correct multiset 
and the induced partition $\pi_i''$ of $\CA_i''$ is a factorization, 
cf.\  Remark \ref{rem:dim2}.
By $\bar H_i$ we indicate the image of $H_i$ under a restriction map.

\begin{table}[ht!b]
\renewcommand{\arraystretch}{1.5}
\begin{tabular}{lllll}\hline
$\pi_i'$ &  $\exp\CA_i'$ & $\alpha_{H_i}$ & $\pi_i''$ & $\exp\CA_i''$\\ 
\hline\hline
$\varnothing$ &  $0, 0, 0$ &  $x$ & $\varnothing$ & $0,0$ \\
$\{H_1\}$ &  $0, 0, 1$ &  $y$ & $\{\bar H_1\}$ & $0,1$ \\
$\{H_1\}, \{H_2\}$ &  $0, 1, 1$ &  $z$ & $\{\bar H_1\}, \{\bar H_2\}$ & $1,1$ \\
$\{H_1\}, \{H_2\}, \{H_3\}$ &  $1, 1, 1$ &  $x+y$ & $\{\bar H_1\}, \{\bar H_3\}$ & $1,1$ \\
$\{H_1\}, \{H_2, H_4\}, \{H_3\}$ &  $1, 1, 2$ &  $x+y-z$ & $\{\bar H_1\}, \{\bar H_2, \bar H_4\}$ & $1,2$ \\
$\{H_1\}, \{H_2, H_4\}, \{H_3, H_5\}$ &  $1, 2, 2$ && \\ 
\hline
\end{tabular}
\smallskip
\caption{Induction Table of Factorizations for $\CA$}
\label{table0}
\end{table}

Example \ref{ex:notfactored} shows that the order 
of the last two hyperplanes can't be reversed.
The fact that $\CA$ is inductively factored can also be 
deduced directly from 
Proposition \ref{prop:ssindfactored}. For, one checks that
\[
\BBK^3 < H_1 < H_1 \cap H_2 \cap H_4 < \{0\}
\]
is a maximal chain of modular elements in $L(\CA)$.
\end{example}

Our next example illustrates 
another induction table of factorizations as explained in
Remark \ref{rem:indtable}(ii).
This example shows that 
the converse of Proposition \ref{prop:ssindfactored} is false.

\begin{example}
\label{ex:d13}
Let $\CA = \CD_3^1$ be the $3$-arrangement with defining polynomial  
\[
Q(\CD_3^1) = x (x - y) (x + y) (x - z) (x + z) (y - z) (y + z).
\] 
Note that $\CA$ is the restriction of a Coxeter arrangement of 
type $D_4$ to a hyperplane, cf.\ \cite[Ex.\ 6.83]{orlikterao:arrangements}.
By \cite[Ex.\ 5.5]{jambuterao:free}, $\CA$ is not supersolvable. 
In particular, $L(\CA)$ does not admit a modular element of
rank $2$, cf.\ \cite[Lem.\ 2.5]{hogeroehrle:super}.
For simplicity, we enumerate the seven hyperplanes of $\CA$ in the order their linear forms appear as factors in $Q(\CA)$, 
i.e.\ $H_1 = \ker(x), H_2 = \ker(x-y), \ldots, H_7 = \ker(y + z)$.
Thanks to \cite[Ex.~2.6]{jambuterao:free}, 
$\CA$ is inductively free with $\exp \CA = \{1,3,3\}$.
We claim that 
\[
\pi = (\pi_1, \pi_2, \pi_3) = (\{H_1\}, \{H_2, H_3, H_6\}, \{H_4, H_5, H_7\})
\]
is an inductive factorization of $\CA$.
We show this 
using the Addition-Deletion Theorem \ref{thm:add-del-free-factored}
by exhibiting an induction table of factorizations for $\CA$, cf.\ Remark \ref{rem:indtable}(ii).
By Remark \ref{rem:dim2} and Proposition \ref{prop:product-indfactored},
$\CA_3:= \CD_2^1 \times \Phi_1$ is an inductively factored subarrangement of 
$\CA$ with defining polynomial $x(x-y)(x+y)$. 
We start our induction table with $\CA_3$.
We indicate the corresponding inductive factorizations of $\CA_i'$
and $\CA_i''$
in the columns headed by $\pi_i'$ and $\pi_i''$, respectively. 
At each step, we have to verify 
the conditions in Theorem \ref{thm:add-del-free-factored}(iii): i.e.,  that 
$\exp \CA_i''$ is the correct multiset 
and that the induced partition $\pi_i''$ of $\CA_i''$ is a factorization.
 
For brevity, we denote a
partition of $\CA$ such as  
$(\{H_1\}, \{H_2, H_3, H_6\}, \{H_4, H_5, H_7\})$,
simply by the sets of the corresponding indices, 
i.e., $\{1\}, \{2, 3, 6\}, \{4, 5, 7\}$.
Also, the notation $\bar1$ indicates the image $\R_i(H_1)$ in $\pi_i''$.  

\begin{table}[ht!b]
\renewcommand{\arraystretch}{1.5}
\begin{tabular}{lllll}\hline
$\pi_i'$ &  $\exp\CA_i'$ & $\alpha_{H_i}$ & $\pi_i''$ & $\exp\CA_i''$\\ 
\hline\hline
$\{1\}, \{2, 3\}$ &  $0, 1,2$ &  $x - z$ & $\{\bar1\}, \{\bar2, \bar3\}$ & $1,2$ \\
$\{1\}, \{2, 3\},\{4\}$ & $1, 1,2$ & $x + z$ & $\{\bar1\}, \{\bar2, \bar3\}$ & $1,2$ \\
$\{1\}, \{2, 3\},\{4,5\}$ & $1,2, 2$ & $y-z$ & $\{\bar1\}, \{\bar4, \bar5\}$ & $1,2$ \\
$\{1\}, \{2, 3, 6\},\{4,5\}$ & $1,2,3$ & $y+z$ & $\{\bar1\}, \{\bar2, \bar3, \bar6\}$ &$1,3$ \\
$\{1\}, \{2, 3, 6\},\{4,5,7\}$ & $1,3, 3$ && \\ 
\hline
\end{tabular}
\smallskip
\caption{Induction Table of Factorizations for $\CD^1_3$}
\label{table1}
\end{table}

The exponents in Table \ref{table1} can be determined using 
Theorem \ref{thm:add-del} and \cite[Prop.\ 6.82]{orlikterao:arrangements}.
Note, as a $2$-arrangement each $\CA_i''$ is  inductively factored,
by Remark \ref{rem:dim2}.

Noteworthy is also that this example illustrates
that there is no obvious addition-deletion theorem for 
supersolvable arrangements,
as both $\CA'$ and $\CA''$ are supersolvable while
$\CA$ is not. 
\end{example}

Our next example 
(which was already noted by Jambu and Paris 
\cite{jambuparis:factored})
demonstrates that 
a free and factored arrangement need not be 
inductively factored. 
Moreover, it shows that 
the condition on $\exp \CA''$ in 
Theorem \ref{thm:add-del-free-factored}(iii)
is necessary for the ``Deletion'' statement.
In the second part of this example we illustrate that 
even if the underlying arrangement is inductively factored,
``Addition'' in Theorem \ref{thm:add-del-factored}
does not necessarily yield an inductive factorization, 
cf.\  Remark \ref{rem:inducednotindfree}.

\begin{example}
\label{ex:g333}
Let $\CA$ be the reflection arrangement 
of the monomial group $G(3,3,3)$.
Let $x, y$ and $z$ be the indeterminates of $S$ and let 
$\zeta = e^{2\pi i/3}$ be a primitive $3$rd root of unity. 
Then the defining polynomial of $\CA$ is given by 
\[
Q(\CA) = (x-y)(x-\zeta y)(x-\zeta^2 y)(x-z)(x-\zeta z)(x-\zeta^2 z)(y-z)(y-\zeta z)(y-\zeta^2 z).
\]
For simplicity, we enumerate the nine hyperplanes of $\CA$ in the order their linear forms appear as factors in $Q(\CA)$, 
i.e.\ $H_1 = \ker(x-y), \ldots, H_9 = \ker(y-\zeta^2 z)$.
One checks that 
\[
\pi = (\pi_1, \pi_2, \pi_3) = (\{H_1,H_2, H_4, H_5\}, \{H_7\}, \{H_3, H_6, H_8, H_9\})
\]
is a factorization of $\CA$. 

Noting that $G(3,3,3)$ acts transitively on $\CA$, \cite[\S 6.4]{orlikterao:arrangements},
forming a triple with respect to any choice of a hyperplane in $\CA$ gives
the following exponents
\begin{equation}
\label{eq:g333}
\exp \CA'' = \{1,3\} \not\subseteq \exp \CA = \{1,4,4\}.
\end{equation}
One checks that for any choice of hyperplane $H_0$ in $\pi_1$ or $\pi_3$ above,
$\pi''$ is a factorization of $\CA''$. As a result of 
Remark \ref{rem:exponents} and 
\eqref{eq:g333}, the condition on $\exp \CA''$ in 
Theorem \ref{thm:add-del-free-factored}(iii) is never satisfied.
Thus, although both 
$\CA$ and $\CA''$ are factored and free, there is no distinguished 
hyperplane in $\CA$ with respect to $\pi$.
In particular, $\CA$ is not inductively factored.
Indeed, as a consequence of \eqref{eq:g333}, $\CA'$ is not free 
and thus $\CA$ is not inductively free, 
cf.\ \cite[Ex.\ 2.19]{hogeroehrle:indfree}.

Now we apply Theorem \ref{thm:add-del-factored} 
to obtain a nice partition of an extended arrangement.
For that, set $H_{10} := \ker x$ and  
$\hat\CA = \CA \cup \{H_{10}\}$. 
Let $\hat \pi = (\hat \pi_1, \hat \pi_2, \hat \pi_3)  = (\pi_1 \cup \{ H_{10}\}, \pi_2, \pi_3)$ and let
$(\hat\CA, \hat\CA' = \CA, \hat\CA'')$ be the triple
associated with $H_{10}$.
One checks that 
the restriction map $\R = \R_{\hat \pi, H_{10}} \colon \hat\CA \setminus \hat\pi_1 \rightarrow \hat\CA''$ is 
bijective and $\hat \pi'' = (\R(\hat\pi_2), \R(\hat\pi_3))$ is a factorization of $\hat\CA''$. 
Therefore,  by Theorem \ref{thm:add-del-factored},
\[
\hat \pi = (\{H_1,H_2, H_4, H_5, H_{10}\}, \{H_7\}, \{H_3, H_6, H_8, H_9\}) 
\]
is a nice partition of $\hat\CA$.
One can show that $\hat\CA$ is inductively factored, the proof is
similar to the one for $\CD_3^1$ in Example \ref{ex:d13}.
However, $\hat\pi$ is not an inductive factorization of $\hat\CA$.
For, deleting $H_{10}$ leads to the non-inductive factorization $\pi$ of $\CA$ and 
the only other hyperplanes $H$ that can be deleted from $\hat\CA$ such that $\hat\CA\setminus\{H\}$ is free
are $H_7, H_8$ and $H_9$. In each case the 
exponents of $\hat\CA\setminus\{H\}$ are $\{1,4,4\}$. 
However, removing any of these hyperplanes from $\hat \pi$ results in 
a partition $\hat \pi\setminus\{H\}$ whose parts have cardinalities 
$\{4,5\}$, $\{1,3,5\}$ and $\{1,3,5\}$, respectively. 
It thus follows from Remark \ref{rem:exponents} that  
$\hat \pi\setminus\{H\}$  is not nice for $\hat\CA\setminus\{H\}$.
In particular, $\hat \pi$ can't be an inductive factorization of $\hat\CA$.
\end{example}

As a further application of our
Addition-Deletion Theorem \ref{thm:add-del-free-factored}, 
we present two examples of an arrangement $\CA$ which 
is even inductively free and factored but not inductively factored.
Our first example also gives an instance where there is a unique nice 
partition of $\CA$.
While the first example can't be realized as the intersection lattice
of an arrangement over $\BBC$, the second 
one is based on 
the lattice of the reflection arrangement $\CA(G(3,3,3))$
from Example \ref{ex:g333}.

\begin{example}
\label{ex:indfreefactored-notindfactored}
Let $\BBK = \BBF_4$ be the field of $4$ elements and let
$\zeta$ be a primitive third root of unity in $\BBK$, 
so that  $\zeta^2 + \zeta + 1 = 0$.
Let $V = \BBK^3$ and let
$x, y$ and $z$ be the coordinate functions of $V$. 
Let $\CA$ be the arrangment in $V$ given by 
\begin{align*}
Q(\CA) = & \ (x-y)(x-\zeta y)(x-z)(x-\zeta^2 y)(y-\zeta z)y(y-z)\\
& \ (x+y+\zeta^2 z)(y-\zeta^2 z)(x-\zeta z)(x-\zeta^2 z).
\end{align*}
For simplicity, we enumerate the eleven
hyperplanes of $\CA$ in the order their 
linear forms appear as factors in $Q(\CA)$, 
i.e.\ $H_1 = \ker(x-y), \ldots, H_{11} = \ker(x-\zeta^2 z)$.
One checks that $\CA$ is inductively free with
$\exp \CA = \{1,4,6\}$, where  
an inductive chain of hyperplanes in $\CA$ is given by 
$H_1, \ldots, H_{11}$ in this order, cf.\ 
Remark \ref{rem:indtable}(i).

Moreover, one checks that 
\[
\pi = (\pi_1, \pi_2, \pi_3) = 
(\{H_{11}\}, \{H_1,H_3, H_9, H_{10}\}, \{H_2, H_4, H_5, H_6, H_7, H_8\})
\]
is a factorization of $\CA$ and one can show 
that this is the only nice partition of $\CA$.
We omit the details.

However, $\pi$ is not an inductive factorizaton of $\CA$.
For, a deletion $\CA\setminus \{H\}$ is free if and only if 
$H \in \{H_1, H_3, H_6, H_8, H_9, H_{10}, H_{11}\}$.
For each of these instances we have 
$|\CA^H| = 5$ and consequently 
$\exp \CA^H = \{1,4\}$ and 
$\exp \CA\setminus\{H\} = \{1,4,5\}$.
Therefore, for $\CA\setminus\{H\}$ to be inductively factored,  
the Addition-Deletion Theorem \ref{thm:add-del-free-factored}
implies that $H$ must belong to $\pi_3$, i.e.\
$H = H_6$ or $H = H_8$.
Using the deletion part of 
Theorem \ref{thm:add-del-free-factored} twice, 
we can remove both $H_6$ and $H_8$ from $\pi_3$
and so we obtain a factorization of $\CA\setminus\{H_6, H_8\}$.
However, this resulting arrangement is not 
inductively free 
(it is lattice isomorphic to $\CA(G(3,3,3))$
from Example \ref{ex:g333})
and so it is not inductively factored,
by Proposition \ref{prop:indfactoredindfree}.
As a result, $\CA$ is not inductively factored either.

Finally, we observe 
that the intersection lattice $L(\CA)$ of the 
arrangment $\CA$ above cannot be realized over $\BBC$.
For a contradiction, suppose that $\CB$ is a complex $3$-arrangement
with $L(\CB)$ lattice isomorphic to $L(\CA)$. 
Then in particular, $\CB$ is 
inductively free with $\exp \CB = \{1,4,6\}$.
Then, since $\CB$ is free with $\exp \CB = \{1,4,6\}$, 
it follows from \cite[Thm.\ 11]{ziegler:multiarrangements}
(cf.\ \cite[Thm.\ 1.34(ii)]{yoshinaga:free})
that each multi-restriction of $\CB$ is free with exponents $\{4,6\}$.
The multi-restriction of $\CB$ with respect to a fixed $H$ from $\CB$ 
is given by $(\CB^H, (2,2,2,2,2))$, where 
$(2,2,2,2,2)$ are the multiplicities of the
five hyperplanes in $\CB^H$.
Thanks to \cite[Thm.\ 1.23(iii)]{yoshinaga:free},
the exponents of $(\CB^H, (2,2,2,2,2))$ are $\{5,5\} \ne \{4,6\}$,
a contradiction.
Consequently, 
there is no $3$-arrangement $\CB$  over $\BBC$
whose lattice is isomorphic to $L(\CA)$.
\end{example}

\begin{example}
\label{ex:indfreefactored-notindfactored2}
Let $\BBK=\BBC$ and let $\zeta$ be a primitive third root of unity. 
Let $V= \BBC^3$ and let $x,y$ and $z$ be the coordinate functions of $V$.
Let $\CA$ be the arrangement in $V$ given by
\begin{align*}
Q(\CA) = & \ x(x-y)(x-\zeta y)(x-\zeta^2 y)(x-z)(x-\zeta z)(x-\zeta^2 z)\\
& \ (y-z)(y-\zeta z)(y-\zeta^2 z)(x+y+\zeta^2 z).
\end{align*}
For simplicity, we enumerate the eleven
hyperplanes of $\CA$ in the order their 
linear forms appear as factors in $Q(\CA)$, 
i.e.\ $H_1 = \ker(x), \ldots, H_{11} = \ker(x+y+\zeta^2 z)$.
One checks that $\CA$ is inductively free with
$\exp \CA = \{1,5,5\}$, where  
an inductive chain of hyperplanes in $\CA$ is given by 
$H_1, \ldots, H_{11}$ in this order, cf.\ 
Remark \ref{rem:indtable}(i).

The intersection lattice $L(\CA)$ is determined by the localizations
of rank 2:

$\{H_1,H_2,H_3,H_4\}, \{H_1,H_5,H_6,H_7\}, \{H_1,H_8\}, \{H_1,H_9\}, 
 \{H_1,H_{10}\}, \{H_1,H_{11}\}, \{H_2,H_5,H_8\}, $\\
$\{H_2,H_6,H_9\}, \{H_2,H_7,H_{10}\},\{H_2,H_{11}\}, \{H_3,H_5,H_{10}\},\{H_3,H_6,H_{8}, H_{11}\},
 \{H_3,H_7,H_9\},$ \\ 
$\{H_4,H_5,H_{9},H_{11}\}, \{H_4,H_6,H_{10}\}, \{H_4,H_7, H_{8}\}, \{H_7,H_{11}\},
\{H_8,H_9,H_{10}\}, \{H_{10},H_{11}\}$.

With these at hand it is quite easy to check whether a permutation of the
hyperplanes is a lattice automorphism or when 
a given partition of $\CA$ is actually a factorization.

We claim that $\CA$ has exactly two factorizations. 
To show this assume that $\pi = (\pi_1,\pi_2,\pi_3)$ 
is a factorization of $\CA$. 
The singleton condition in Definition \ref{def:factored}(ii) 
applied to the localizations 
$\{H_1,H_{10}\}$, $\{H_1, H_{11}\}$ and $\{H_{10}, H_{11}\}$ 
shows that $H_1$, $H_{10}$ and $H_{11}$ are in different 
parts of $\pi$. In particular, one of these three
hyperplanes has to be the singleton of $\pi$. 

Using Corollary \ref{cor:teraofactored}(iii) on the localizations 
$\{H_1,H_2,H_3,H_4\}$, $\{H_2,H_{11}\}$, 
$\{H_1,H_5,H_6,H_7\}$ and $\{H_7,H_{11}\}$, we see that $\{H_1\}$  
can't be the singleton. Else all of $H_2,H_3,H_4,H_5,H_6,H_7$, 
$H_{11}$ would 
have to be in one part of $\pi$, but the cardinalities of the 
parts $\pi_i$ are either $1$ or $5$.
There is a lattice automorphism $\sigma$ of $L(\CA)$ induced by 
$(1\;11)(2\;9\;7\;8)(3\;4\;5\;6)$ permuting the hyperplanes. 
Thus $\{H_{11}\}$ cannot be the singleton either.

Hence the only possibility is that $\pi_1 = \{H_{10}\}$ is the singleton 
part of $\pi$. Without loss we may assume that $H_1 \in \pi_2$ and
$H_{11} \in \pi_3$. Applying the singleton condition 
from Definition \ref{def:factored}(ii)
to the localizations 
$\{H_1,H_8\}$, $\{H_1,H_9\}$, $\{H_2,H_{11}\}$ and $\{H_7,H_{11}\}$, 
it follows that $H_8,H_9 \in \pi_3$ and $H_2,H_7 \in \pi_2$. 
Corollary \ref{cor:teraofactored}(iii) applied to
the localizations 
$\{H_{10},H_4,H_6\}$ and $\{H_{10},H_3,H_5\}$ shows that $\{H_4,H_6\}$ 
and $\{H_3,H_5\}$ have to be subsets of one part of $\pi$.

Therefore, the only possibilities for factorizations are 
$$\pi = (\{H_{10}\}, \{H_1,H_2,H_3, H_5, H_7\}, \{H_4,H_6,H_8,H_9,H_{11}\})$$
and 
$$\tilde \pi = (\{H_{10}\},\{H_1,H_2,H_4,H_6,H_7\},\{H_3,H_5,H_8,H_9,H_{11}\}).$$
One checks that both $\pi$ and $\tilde \pi$ are indeed factorizations of $\CA$.

The permutation $(2\;7)(3\;6)(4\;5)$ induces an automorphism of $L(\CA)$ 
by permuting the corresponding hyperplanes. This automorphism 
interchanges $\pi$ and $\tilde \pi$. Therefore,
to show that none of 
$\pi$ and $\tilde \pi$ is inductive, 
it suffices to consider the factorization $\pi$.

So we aim to show that $\pi$ is not inductive.
A deletion $\CA\setminus\{H\}$ is free if and only if 
$H \in \{H_1, H_{10}, H_{11}\}$. We must not remove $H_{10}$, since this
would contradict the singleton condition.
However, using the Addition-Deletion Theorem \ref{thm:add-del-free-factored}, 
we can remove $H \in \{H_1,H_{11}\}$ and so we get a factorization 
$\pi'$ of $\CA \setminus\{H\}$ since the corresponding restriction maps $\R$ 
are bijective. Since the automorphism $\sigma$
preserves the factorization $\pi$ (and interchanges $H_1$ and $H_{11}$), 
we only have to consider the case when $H = H_{11}$.
Now let $\CB = \CA \setminus \{H_{11}\}$ and let 
$\pi' = (\{H_{10}\}, \{H_1,H_2,H_3, H_5, H_7\}, \{H_4,H_6,H_8,H_9\})$ be
the corresponding factorization.
A deletion $\CB\setminus\{H\}$ is free if and only if 
$H \in \{H_1, H_8,H_9,H_{10}\}$. We mustn't remove $H_{10}$, since this
would contradict the singleton condition 
in Definition \ref{def:factored}(ii) again. 
Likewise $H_8$ and $H_9$ can't be removed, since
the Addition-Deletion Theorem \ref{thm:add-del-free-factored} 
would  then  imply 
that the corresponding restriction map $\R$ is bijective. But this is not
the case, for, if we remove $H_8$ we have $\R(H_2) = \R(H_5)$, else if we
remove $H_9$ we have $\R(H_3) = \R(H_7)$.
But if we remove $H=H_1$ we get a factorization of $\CA(G(3,3,3))$ which is not 
inductively free and therefore not inductively factored.
\end{example}

\subsection{Hereditarily factored Arrangements}
\label{sec:heredfactored}

In analogy to hereditary freeness and hereditary inductive freeness, 
\cite[Def.\ 4.140, p.\ 253]{orlikterao:arrangements}, 
we extend the notions of factored and inductively factored arrangements to 
incorporate restrictions.

\begin{defn}
\label{def:heredindfactored}
We say that $\CA$ is \emph{hereditarily factored}
provided $\CA^X$ is factored for every $X \in L(\CA)$
and that $\CA$ is \emph{hereditarily inductively factored}
provided $\CA^X$ is inductively factored for every $X \in L(\CA)$.
We also use the acronyms $\CHFAC$ and $\CHIFAC$ for short for these classes,
respectively.
\end{defn}

Proposition \ref{prop:ssindfactored} readily strengthens as follows.

\begin{corollary}
\label{cor:ssheredindfactored}
If $\CA$ is supersolvable, then it is 
hereditarily inductively factored.
\end{corollary}

\begin{proof}
If $\CA$ is supersolvable, then so is 
$\CA^X$ for every 
$X \in L(\CA)$, thanks to 
\cite[Prop.\ 3.2]{stanley:super}.
Thus $\CA^X$ is inductively factored,
by Proposition \ref{prop:ssindfactored}
\end{proof}

Also, Proposition \ref{prop:indfactoredindfree}
extends to restrictions. For, if
$X \in L(\CA)$ and $\CA^X$ is inductively factored, 
then $\CA^X$ is inductively free, 
by Proposition \ref{prop:indfactoredindfree}.

\begin{corollary}
\label{cor:heredindfactoredheredindfree}
If $\CA$ is hereditarily inductively factored,
then it is hereditarily inductively free.
\end{corollary}

While a $2$-arrangement is always inductively factored,
by Remark \ref{rem:2-arr},  
in general, a factored $3$-arrangement need not be inductively factored,
see Example \ref{ex:g333}.
Nevertheless, for a $3$-arrangement, 
we have the following 
counterpart to \cite[Lem.\ 2.15]{hogeroehrle:indfree}
in our setting.

\begin{lemma}
\label{lem:3-arr}
Suppose that $\ell = 3$. Then 
$\CA$ is (inductively) factored if and only if it is 
hereditarily (inductively) factored.
\end{lemma}

\begin{proof}
The reverse implication is clear.
For the converse, 
if $\CA$ is (inductively) factored and $X \in L(\CA)\setminus \{V\}$,
then $\dim X \le 2$.
The result then follows from Remark~\ref{rem:dim2}.
\end{proof}

Our next example shows that the equivalence from 
Lemma \ref{lem:3-arr} already fails in dimension~$4$.

\begin{example} 
\label{ex:notheredfactored}
Let $\BBK = \BBQ$ be the field of rational numbers.
Let $\CA$ be the $4$-arrangement defined by the 
ten forms $\alpha_{H_i}$ shown in 
column three of Table \ref{indtable:non-HIFAC},
where we denote the coordinate functions in $S$ simply 
by $x, y, z$ and $t$.

In Table \ref{indtable:non-HIFAC} we present an 
induction table of factorizations for $\CA$ which 
shows that $\CA$ is inductively factored.
In \cite[Ex.\ 2.16]{hogeroehrle:indfree}, we showed that 
$\CA$ is inductively free with $\exp \CA = \{1,3,3,3\}$, but
that $\CA$ is not 
hereditarily inductively free.  
Therefore, it follows from 
Corollary \ref{cor:heredindfactoredheredindfree}
that $\CA$ is not hereditarily inductively factored.
Note that the induction table in \cite[Ex.\ 2.16]{hogeroehrle:indfree}
used a different ordering of the hyperplanes.

We enumerate the ten hyperplanes in the order they appear in 
Table \ref{indtable:non-HIFAC}, 
i.e.\  $H_1 = \ker(t), \ldots$, $H_{10} = \ker(z)$.
We claim that 
\[
\pi = (\pi_1, \ldots, \pi_4) 
= (\{H_1\},\{H_2,H_3,H_{10}\},\{H_4,H_5,H_9\}, \{H_6,H_7,H_8\})
\]
is an inductive factorization of $\CA$.
This follows from the data in Table \ref{indtable:non-HIFAC},
the Addition-Deletion Theorem \ref{thm:add-del-free-factored},
along with the fact that
each occurring restriction $\CA_i''$ is itself 
again inductively factored with the given set of exponents
which can be checked separately. 
As in Table \ref{table1}, 
we indicate the corresponding inductive factorizations of $\CA_i'$
and $\CA_i''$
in the columns headed by $\pi_i'$ and $\pi_i''$, respectively. 
One needs to check that $\exp \CA_i''$ is the required 
multiset and that the induced partition $\pi_i''$
of $\CA_i''$ is a factorization.
As in Table \ref{table1} above, we represent  
the inductive factorizations $\pi_i'$ of the 
subarrangements $\CA_i'$ simply by listing 
the indices of the respective hyperplanes, and  
the notation $\bar1$ indicates the image 
of $H_1$ in $\pi_i''$.  

\begin{table}[h]
\renewcommand{\arraystretch}{1.5}
\begin{tabular}{lllll}  \hline
$\pi_i'$ &  $\exp \CA_i'$ & $\alpha_{H_i}$ & $\pi_i''$ & $\exp \CA_i''$\\
\hline\hline
$\varnothing$ & $0, 0, 0, 0$ &  $t$ & $\varnothing$  & $0,0,0$ \\
$\{1\}$       & $0, 0, 0, 1$ &  $x+y-z+t$& $\{\bar1\}$	& $0,0,1$ \\
$\{1\},\{2\}$ & $0, 0, 1, 1$ &  $x+y-z-t$ & $\{\bar1\}$  & $0,0,1$ \\    
$\{1\},\{2,3\}$    & $0, 0, 1, 2$ & $x-y+z-t$ & $\{\bar1\},\{\bar2,\bar3\}$ & $0,1,2$ \\
$\{1\},\{2,3\},\{4\}$ & $0, 1, 1, 2$ & $x-y+z+t$ & $\{\bar1\},\{\bar2,\bar3\}$ & $0,1,2$ \\
$\{1\},\{2,3\},\{4,5\}$ & $0, 1, 2, 2$ & $x+y+z+t$ & $\{\bar1\},\{\bar2,\bar3\},\{\bar4,\bar5\}$  & $1,2,2$ \\    
$\{1\},\{2,3\},\{4,5\}, \{6\}$ & $1, 1, 2, 2$ & $x+y+z-t$ & $\{\bar1\},\{\bar2,\bar3\},\{\bar4,\bar5\}$  & $1,2,2$  \\    
$\{1\},\{2,3\},\{4,5\}, \{6,7\}$  & $1, 2, 2, 2$&  $y$ & $\{\bar1\},\{\bar2,\bar3\},\{\bar4,\bar5\}$  & $1,2,2$  \\
$\{1\},\{2,3\},\{4,5\}, \{6,7,8\}$   & $1, 2, 2, 3$ &  $x$ & $\{\bar1\},\{\bar2,\bar3\},\{\bar6,\bar7, \bar8\}$   & $1,2,3$ \\        
$\{1\},\{2,3\},\{4,5,9\}, \{6,7,8\}$   & $1, 2, 3, 3$  & $z$ & $\{\bar1\},\{\bar4,\bar5, \bar9\},\{\bar6,\bar7, \bar8\}$   & $1,3,3$ \\
$\{1\},\{2,3,10\},\{4,5,9\}, \{6,7,8\}$  & $1, 3, 3, 3$ & & \\ \hline 
\end{tabular}\\
\bigskip
\caption{Induction Table of Factorizations for $\CA \in \CIFAC\setminus \CHIFAC$.} 
\label{indtable:non-HIFAC} 
\end{table}
\end{example}

\subsection{Products of factored Arrangements}
\label{sec:products}

In this section we show that the various notions of factorizations
from the previous sections 
are compatible with the product construction for arrangements.

\begin{proposition}
\label{prop:product-factored}
Let $(\CA_1, V_1), (\CA_2, V_2)$ be two arrangements.
Then  $\CA = (\CA_1 \times \CA_2, V_1 \oplus V_2)$ is 
nice if and only if both 
$(\CA_1, V_1)$ and $(\CA_2, V_2)$ are nice.
\end{proposition}

\begin{proof}
First assume that both $\CA_1$ and $\CA_2$ are nice.
Let $\pi^i = (\pi_1^i, \ldots, \pi_{r_i}^i)$ 
be a nice partition of $\CA_i$,
where $r_i$ is the rank of $\CA_i$ for $i = 1,2$.
Then $r = r_1 + r_2$ is the rank of $\CA$.
Define a partition  
$\pi = (\pi_1, \ldots, \pi_{r_1}, \pi_{r_1+1}, \ldots, \pi_r)$
of $\CA$ by setting
$\pi_j = \{ H_1 \oplus V_2 \mid H_1 \in \pi_j^1\}$ for $1 \le j \le r_1$ and 
$\pi_{j+r_1} = \{ V_1\oplus H_2 \mid H_2 \in \pi_j^2\}$ for $1 \le j \le r_2$.
We claim that $\pi$ is nice for $\CA$.

Since $\pi^1, \pi^2$ are independent, 
so is $\pi$, by construction.
Let $X = X_1  \oplus X_2 \in L(\CA) \setminus \{V_1 \oplus V_2\}$,
cf.\ \eqref{eq:latticesum}.
Without loss, assume that $X_1 \ne V_1$ (else $X_2 \ne V_2$.)
Then, since $\pi^1$ is a factorization of
$\CA_1$, there is a block $\pi_k^1$ such that 
$\pi_k^1 \cap \CA_{X_1} =\{H_1\}$ is a singleton, and thus, 
by definition of $\pi$ and 
using \eqref{eq:product} and \eqref{eq:productAX}, we see that
\[
\pi_k \cap \CA_X = (\pi_k^1 \cap \CA_{X_1}) \oplus V_2 = \{H_1 \oplus V_2\}
\] 
is a singleton, as required. 

Conversely, suppose that $\pi$ is a factorization of $\CA$.
We define partitions $\pi^i$ of  $\CA_i$ for $i = 1,2$ 
simply by taking as blocks the hyperplanes in 
$\CA_i$ that occur as 
a direct summand of a 
member in a block of $\pi$. 
More precisely, the non-empty 
$\pi_j^1 = \{ H_1 \in \CA_1 \mid H_1 \oplus V_2 \in \pi_j \}$ 
are the blocks of $\pi^1$; $\pi^2$ 
is defined analogously.

We claim that $\pi^i$ is a factorization of $\CA_i$ for $i = 1,2$.
We show this for $\pi^1$, the proof for $\pi^2$ is completely analogous.

Since $\pi$ is independent, so is $\pi^1$.
Let $X_1 \in L(\CA_1)\setminus\{V_1\}$. Then $X = X_1 \oplus V_2$ belongs to 
$L(\CA) \setminus \{V_1 \oplus V_2\}$,
cf.\ \eqref{eq:latticesum}.
Since $\pi$ is a factorization of $\CA$, 
there is a block, $\pi_k$ say, such that 
$\pi_k \cap \CA_X = \{H_1 \oplus V_2\}$ is a singleton.
But then, 
using \eqref{eq:product} and \eqref{eq:productAX} again, 
$\pi_k^1 \cap \CA_{X_1} = \{H_1\}$, as required.
Consequently, $\pi^1$ is nice for $\CA_1$, as claimed.
\end{proof}

We can strengthen  
Proposition \ref{prop:product-factored}
further and restrict
the compatibility with products  
to the class of inductively factored arrangements.

\begin{proposition}
\label{prop:product-indfactored}
Let $(\CA_1, V_1), (\CA_2, V_2)$ be two arrangements.
Then  $\CA = (\CA_1 \times \CA_2, V_1 \oplus V_2)$ is 
inductively factored if and only if both 
$(\CA_1, V_1)$ and $(\CA_2, V_2)$ are 
inductively factored  and in that case
the multiset of exponents of $\CA$ is given by 
$\exp \CA = \{\exp \CA_1, \exp \CA_2\}$.
\end{proposition}

\begin{proof}
First suppose that both 
$\CA_1$ and $\CA_2$ are inductively factored. 
As both are factored, so is $\CA$, 
by Proposition \ref{prop:product-factored}.
We show that $\CA = \CA_1 \times \CA_2$ is inductively factored
by induction on $n = |\CA| = |\CA_1| + |\CA_2|$. For 
$n = 0$ we have $\CA = \Phi_\ell$ and there is nothing to prove.
Now suppose that $n \ge 1$ and that the result holds for 
products of inductively factored arrangements with less than $n$ hyperplanes.
Let $\pi^i = (\pi_1^i, \ldots, \pi_{r_i}^i)$ 
be an inductive factorization of  $\CA_i$, 
where $r_i$ is the rank of $\CA_i$ for $i = 1,2$.
Then $r = r_1 + r_2$ is the rank of $\CA$.
As in the proof of Proposition \ref{prop:product-factored}, define a partition  
$\pi = (\pi_1, \ldots, \pi_{r_1}, \pi_{r_1+1}, \ldots, \pi_r)$
of $\CA$ by setting
$\pi_j = \{ H_1 \oplus V_2 \mid H_1 \in \pi_j^1\}$ for $1 \le j \le r_1$ and 
$\pi_{j+r_1} = \{ V_1\oplus H_2 \mid H_2 \in \pi_j^2\}$ for $1 \le j \le r_2$.
By the argument in the proof of Proposition \ref{prop:product-factored}, $\pi$ is a factorization of $\CA$.

Without loss, 
we may assume that $|\CA_1|  > 0$ and that 
there is a hyperplane $H_1 \in \pi^1_1$ 
which is distinguished with respect to $\pi^1$
so that 
$\CA_1' = \CA_1 \setminus \{H_1\}$ and 
$\CA_1'' = \CA_1^{H_1}$ are  inductively factored.
Let  $H_0 = H_1  \oplus V_2 \in \CA$. 
Let $(\CA, \CA', \CA'')$ be the triple with respect to $H_0$. 
Let $\pi' = \pi \cap \CA'$ be the induced partition.
Then for the parts of $\pi'$ we have
\begin{equation}
\label{eq:pi}
\pi'_1 = \pi_1\setminus \{H_0\} \quad \text{ and }
\quad \pi'_i = \pi_i \quad  \text {for }  2 \le i \le r.
\end{equation}

We claim that 
$H_0$ is distinguished with respect to 
$\pi$. So we need to show that 
$\pi'$ is a factorization of $\CA'$.
Since $\pi$ is independent, clearly so is $\pi'$.

Let $X \in L(\CA')\setminus \{V_1 \oplus V_2\}$. 
Then $X = X_1 \oplus X_2$, where $X_1 \in L(\CA_1')$
and $X_2 \in L(\CA_2)$, cf.\ \eqref{eq:product}  and \eqref{eq:latticesum}.
If $X_2 \ne V_2$, then by assumption on $\pi^2$, there
exists a $\pi^2_j$ for some $1 \le j \le r_2$ so that
$\pi^2_j \cap (\CA_2)_{X_2}$ is a singleton.
Thus, by \eqref{eq:productAX} and \eqref{eq:pi}, 
\[
\pi'_{j + r_1} \cap (\CA')_X = \pi_{j + r_1} \cap (\CA')_X 
= V_1 \oplus \left(\pi^2_j \cap (\CA_2)_{X_2}\right)
\]
is a singleton.
Now suppose that $X_2 = V_2$. Then $X_1 \ne V_1$.
Consequently, since $H_1$ is distinguished for $\pi^1$, 
there is a 
$\pi^1_k$ for some $1 \le k \le r_1$ so that
$\pi^1_k \cap (\CA_1')_{X_1} = \{H'\}$ is a singleton.
Since $H_1 \notin \CA_1'$, we have $H' \ne H_1$.
Consequently, we have 
\[
\pi'_k \cap (\CA')_X = \pi_k \cap (\CA')_X = \left(\pi^1_k \cap (\CA_1')_{X_1}\right) \oplus V_2 = \{H' \oplus V_2\}
\]
is a singleton.

Thanks to \eqref{eq:product} and \eqref{eq:restrproduct}
and our results above,
$\CA' = \CA_1' \times \CA_2$ and 
$\CA'' = \CA_1'' \times \CA_2$ are both 
products of inductively factored arrangements.
Therefore, since $|\CA'|, |\CA''| < n$,
it follows from our induction hypothesis that 
both $\CA'$ and $\CA''$ are inductively factored.
As $H_0$ is distinguished 
with respect to $\pi$, it
follows from Proposition \ref{prop:JPinduced} that
$\R : \CA\setminus \pi_1 \to \CA''$ is bijective and thus
$\CA$ is inductively factored, as desired.

Conversely, suppose that $\pi$ is an inductive factorization of $\CA$. 
Then by Proposition \ref{prop:product-factored}, both 
$\CA_1$ and $\CA_2$ are factored.
As in the proof of Proposition \ref{prop:product-factored},
we define partitions $\pi^i$ of  $\CA_i$ for $i = 1,2$ simply by taking as blocks the hyperplanes in 
$\CA_i$ that 
occur as a direct summand of a member 
in a block of $\pi$. More precisely, the non-empty 
$\pi_j^1 = \{ H_1 \in \CA_1 \mid H_1 \oplus V_2 \in \pi_j \}$ are the blocks of $\pi^1$; $\pi^2$ 
is defined analogously.

We show that both 
$\CA_1$ and $\CA_2$ are inductively factored again by induction on $n = |\CA|$.
If $n = 0$, then $\CA = \Phi_\ell$ and so both $\CA_1$ and $\CA_2$
are empty and there is nothing to show.
So suppose that $n \ge 1$  and that the result holds for 
products which are inductively factored and have 
less than $n$ hyperplanes. Since $\CA$ is inductively factored, 
there is a distinguished hyperplane $H_0$ in $\CA$, so that 
the triple  $(\CA, \CA', \CA'')$ 
with respect to $H_0$
is a triple of inductively factored
arrangements.
Without loss, we may assume that 
$H_0 = H_1 \oplus V_2$ is 
distinguished with respect to $\pi$, 
for some $H_1 \in \CA_1$.

We claim that $H_1$ is distinguished with respect to $\pi^1$.
Let $(\CA_1, \CA_1', \CA_1'')$ be the triple associated with $H_1$.
We need to show that $\pi^1 \cap \CA_1'$ is a factorization of $\CA_1'$.
Since $\pi^1$ is independent, so is $\pi^1 \cap \CA_1'$.
Let $X_1 \in L(\CA_1')\setminus\{V_1\}$. Then $X = X_1 \oplus V_2$ belongs to 
$L(\CA') \setminus \{V_1 \oplus V_2\}$, cf.\ \eqref{eq:latticesum}.
Since $\pi$ is an inductive factorization of $\CA$, there is a part $\pi_k$ of $\pi$
so that $\pi_k \cap (\CA')_X$ is a singleton.
Thus, using \eqref{eq:productAX}, we see that
\[
\pi_k \cap (\CA')_X = \left(\pi^1_k \cap (\CA_1')_{X_1} \right) \oplus V_2
\]
and so $\pi^1_k \cap (\CA_1')_{X_1}$ is a singleton, as
desired. 

Thanks to \eqref{eq:product} and \eqref{eq:restrproduct},  both
$\CA' = \CA_1' \times \CA_2$ 
and $\CA'' = \CA_1'' \times \CA_2$
are products. 
Therefore, since $|\CA'|,  |\CA''| < n$ 
and both $\CA'$ and $\CA''$ are inductively factored,
it follows from our induction hypothesis
that also $\CA_1'$, $\CA_1''$ and $\CA_2$
are inductively factored.
Since $H_1$ is distinguished with respect to $\pi^1$, it 
follows from Proposition \ref{prop:JPinduced} that the restriction map 
$\R_{\pi^1, H_1}$ is bijective and thus
$\CA_1$ is inductively factored as well.

The final statement on exponents  follows from 
Proposition \ref{prop:product-free} and the fact
that inductively factored arrangements are free, 
Proposition \ref{prop:indfactoredindfree}.
\end{proof}

The compatibility from Propositions \ref{prop:product-factored} 
and \ref{prop:product-indfactored}
restricts even further to the classes of hereditarily factored and 
hereditarily inductively factored arrangements, respectively.

\begin{corollary}
\label{cor:product-heredindfactored}
Let $\CA_1, \CA_2$ be two arrangements.
Then  $\CA = \CA_1 \times \CA_2$ is 
hereditarily (inductively) factored if and only if both 
$\CA_1$ and $\CA_2$ are 
hereditarily (inductively) factored.
In case of inductively factored arrangements
the multiset of exponents of $\CA$ is given by 
$\exp \CA = \{\exp \CA_1, \exp \CA_2\}$.
\end{corollary}

\begin{proof}
First suppose that both $\CA_1$ and $\CA_2$ are
hereditarily (inductively) factored.
Let $X = X_1 \oplus X_2$ be in $L(\CA)$.
Then, by \eqref{eq:restrproduct} and 
Proposition \ref{prop:product-factored}
(Proposition \ref{prop:product-indfactored}),
$\CA^X = \CA_1^{X_1} \times \CA_2^{X_2}$ is 
(inductively) factored.

Conversely, suppose that 
$\CA$ is hereditarily (inductively) factored.
Let $X_i \in L(\CA_i)$ for $i=1,2$. 
Then $X = X_1 \oplus X_2 \in L(\CA)$.
By \eqref{eq:restrproduct} and 
Proposition \ref{prop:product-factored}
(Proposition \ref{prop:product-indfactored}),
both $\CA_1^{X_1}$ and $\CA_2^{X_2}$ are
(inductively) factored. 

The final statement on exponents  
follows from Propositions \ref{prop:product-free}
and the fact that inductively factored arrangements 
are free, Proposition \ref{prop:indfactoredindfree}.
\end{proof}

\bigskip

Utilizing the results of this note,  in our final comment
we compare various classes of free 
and factored arrangements.

\begin{remark}
\label{rem:inclusions}
For the first purpose of this remark
let $\CSS$, $\CIFAC$, $\CIF$, $\RF$ and $\CF$
denote the classes of supersolvable, %factored, 
inductively factored, inductively free,
recursively free (\cite[Def.\ 4.60]{orlikterao:arrangements}), 
and free arrangements, respectively.
Thanks to Proposition \ref{prop:indfactoredindfree}
(\cite[Prop.\ 2.2]{jambuparis:factored}),
Proposition \ref{prop:ssindfactored} and  the definition of $\RF$,
we have the following containments
\[
\CSS \subseteq \CIFAC \subseteq \CIF \subseteq \RF \subseteq \CF.
\]
It follows from Example \ref{ex:d13} that 
the first inclusion is proper.
Thanks to Remark \ref{rem:indfactored}(ii), the Coxeter arrangement of 
type $D_4$ is not inductively factored but inductively free.
Thus the second inclusion is also proper.

By \cite[Thm.\ 1.1]{hogeroehrle:indfree}, $\CA(G(3,3,3))$ 
is not inductively free.
But thanks to  \cite[Thm.\ 3.6(ii)]{amendhogeroehrle:indfree},
$\CA(G(3,3,3))$ is recursively free. 
Thus the third inclusion is proper.

Terao \cite{terao:freeI} has shown that each reflection
arrangement $\CA(W)$ is free (cf.\ \cite[\S 6]{orlikterao:arrangements}).
In \cite[Rem.\ 3.7]{cuntzhoge},
Cuntz and the first author showed that  the reflection arrangement
$\CA(G_{27})$ 
of the rank $3$ 
exceptional complex reflection group $G_{27}$
is not recursively free.
Consequently, the final inclusion is also proper.

Thus we have proper inclusions throughout:
\[
\CSS \subsetneq \CIFAC \subsetneq \CIF \subsetneq \RF \subsetneq \CF.
\]

It is also of interest 
to compare the hereditary classes from above.
For that purpose let $\CFac$ be the class of nice 
arrangements from  Definition \ref{def:factored}, while  
$\CHFAC$ and $\CHIFAC$ denote the classes 
of hereditarily factored and 
hereditarily inductively factored arrangements
from Definition \ref{def:heredindfactored}. 
Then by definition and 
Corollary \ref{cor:ssheredindfactored}, we have
\[
\CSS  \subseteq \CHIFAC \subseteq \CHFAC \subseteq \CFac. 
\]
It follows from Example \ref{ex:d13} 
and Lemma \ref{lem:3-arr}
that the first inclusion is proper.

Thanks to Example \ref{ex:g333}, 
$\CA = \CA(G(3,3,3))$ is 
factored and thus by 
Lemma \ref{lem:3-arr}, 
$\CA$ belongs to $\CHFAC$.
However, since $\CA$ is not inductively free,
it does not belong to $\CHIFAC$.
Consequently, the second inclusion is proper.

The fact that the final inclusion is proper
follows from Example \ref{ex:notheredfactored}.
For, here one can check that the restriction of $\CA$ to $H_1 = \ker (t)$
gives a $3$-arrangement whose Poincar\'e polynomial 
does not factor into linear terms over $\BBZ$.
It follows from Corollary \ref{cor:teraofactored}(i)
that $\CA^{H_1}$ is not nice and therefore,
$\CA$ is not hereditarily factored.
Therefore, we also have proper inclusions throughout here as well:
\[
\CSS  \subsetneq \CHIFAC \subsetneq \CHFAC \subsetneq \CFac. 
\]
That all these classes of arrangements differ is not 
surprising. Quite striking however is the fact that  
counterexamples to essentially each 
reverse containment are found among the small rank reflection arrangements
and their restrictions.
\end{remark}

%%%%%%%%%%%%%%%%%%%%%%%%%%%%%%%%%%%%%%%%%%%%%%%%%%%%%%%%%%%%%%%%%%%%%%
%%%%%%%%%%%%% Acknowledgments
%%%%%%%%%%%%%%%%%%%%%%%%%%%%%%%%%%%%%%%%%%%%%%%%%%%%%%%%%%%%%%%%%%%%%%

%%%%%%%%%%%%%%%%%%%%%%%%%%%%%%%%%%%%%%%%%%%%%%%%%%%%%%%%%%%%%%%%%%%%%%
%%%%%%%%%%%%% bibliography
%%%%%%%%%%%%%%%%%%%%%%%%%%%%%%%%%%%%%%%%%%%%%%%%%%%%%%%%%%%%%%%%%%%%%%

\bigskip

\bibliographystyle{amsalpha}

\newcommand{\etalchar}[1]{$^{#1}$}
\providecommand{\bysame}{\leavevmode\hbox to3em{\hrulefill}\thinspace}
\providecommand{\MR}{\relax\ifhmode\unskip\space\fi MR }
\providecommand{\MRhref}[2]{%
  \href{http://www.ams.org/mathscinet-getitem?mr=#1}{#2} }
\providecommand{\href}[2]{#2}

%%%%%%%%%%%%%%%%%%%%%%%%%%%%%%%%%%%%%%%%%%%%%%%%%%%%%%%%%%%%%%%%%%%%%%
%%%%%%%%%%%%%%%%%%%%%%%%%%%%%%%%%%%%%%%%%%%%%%%%%%%%%%%%%%%%%%%%%%%%%%

\end{document}